\documentclass[a4paper,11pt]{amsart}
\usepackage{amsthm,amsfonts,amsbsy,amssymb,amsmath,amscd, bm}
\usepackage[all]{xy}
\usepackage[colorlinks=true]{hyperref}
\usepackage{pgf, tikz}

\theoremstyle{plain}
\newtheorem{theorem}{Theorem}[section]
\newtheorem{lemma}[theorem]{Lemma}

\newtheorem{prop}[theorem]{Proposition}
\newtheorem{defi}[theorem]{Definition}

\newtheorem{remark}[theorem]{Remark}

\newtheorem*{ques*}{Question}	
\newtheorem*{conj*}{Conjecture}

\newtheorem*{quesa*}{Question A}
\newtheorem*{quesb*}{Question B}
\newtheorem*{quesc*}{Question C}

\theoremstyle{remark}

\numberwithin{equation}{section}

\newcommand{\QQ}{\mathbb{Q}}

\newcommand{\CC}{\mathbb{C}}
\newcommand{\PP}{\mathbb{P}}
\newcommand{\ZZ}{\mathbb{Z}}
\newcommand{\FF}{\mathbb{F}}

\newcommand{\CE}{\mathcal{E}}

\newcommand{\CL}{\mathcal{L}}
\newcommand{\CO}{\mathcal{O}}

\newcommand{\Pic}{{\mathrm{Pic}}}
\newcommand{\baselocus}{{\mathrm{Bs}}}
\newcommand{\movable}{{\mathrm{Mov}}}
\newcommand{\coeff}{{\mathrm{coeff}}}
\newcommand{\supp}{{\mathrm{Supp}}}

\newcommand{\vol}{{\mathrm{vol}}}

\newcommand{\roundup}[1]{{\left\lceil #1 \right\rceil}}
\newcommand{\rounddown}[1]{{\left\lfloor #1 \right\rfloor}}

\begin{document}

\title{Algebraic threefolds of general type with small volume}
	
\author{Yong Hu}
\author{Tong Zhang}
\date{\today}
	
\address[Y.H.]{School of Mathematical Sciences, Shanghai Jiao Tong University, 800 Dongchuan Road, Shanghai 200240, People's Republic of China}
\email{yonghu@sjtu.edu.cn}
	
\address[T.Z.]{School of Mathematical Sciences, Shanghai Key Laboratory of PMMP, East China Normal University, 500 Dongchuan Road, Shanghai 200241, People's Republic of China}
\email{tzhang@math.ecnu.edu.cn, mathtzhang@gmail.com}
	
\begin{abstract}
	It is known that the optimal Noether inequality $\vol(X) \ge \frac{4}{3}p_g(X) - \frac{10}{3}$ holds for every $3$-fold $X$ of general type with $p_g(X) \ge 11$. In this paper, we give a complete classification of $3$-folds $X$ of general type with $p_g(X) \ge 11$ satisfying the above equality by giving the explicit structure of a relative canonical model of $X$. This model coincides with the canonical model of $X$ when $p_g(X) \ge 23$. We also establish the second and third optimal Noether inequalities for $3$-folds $X$ of general type with $p_g(X) \ge 11$. These results answer two open questions raised by J. Chen, M. Chen and C. Jiang, and in dimension three an open question raised by J. Chen and C. Lai. A novel phenomenon shows that there is a one-to-one correspondence between the three Noether inequalities and three possible residues of $p_g(X)$ modulo $3$. 
\end{abstract}
		




\maketitle

\section{Introduction}

Throughout this paper, we work over the complex number field $\CC$, and all varieties are projective.

\subsection{Motivation} The classical Noether inequality, proved by M. Noether \cite{Noether}, asserts that
\begin{equation} \label{eq: Noether surface}
	K_S^2 \ge 2p_g(S) - 4
\end{equation}
for every minimal surface $S$ of general type. This inequality is one of the foundational results in the theory of surfaces. More importantly, it leads to the so-called \emph{geography of surfaces of general type} which aims to classify surfaces of general type based on the relations of their birational invariants (see \cite[\S 1]{Persson} for more details). As a typical example, minimal surfaces of general type attaining the equality in \eqref{eq: Noether surface} are said to be on the Noether line, and the explicit classification of such surfaces has been completed by E. Horikawa \cite{Horikawa1}.

It is known that the Noether inequality exists in any dimension: there exist $a_n, b_n > 0$ such that $K_X^n \ge a_n p_g(X) - b_n$ for every minimal $n$-fold $X$ of general type \cite{Chen_Jiang}. However, to obtain the optimal $a_n$ and $b_n$ for $n \ge 3$ is quite challenging (see \cite{Kobayashi,Catanese_Chen_Zhang,Chen_Chen} for $n=3$, as well as \cite{Chen_Lai,Totaro_Wang} for general $n$). Recently, J. Chen, M. Chen and C. Jiang \cite{Chen_Chen_Jiang,Chen_Chen_Jiang2} proved the following optimal Noether inequality that
\begin{equation} \label{eq: Noether 3fold}
	K_X^3 \ge \frac{4}{3} p_g(X) - \frac{10}{3}
\end{equation}
for every minimal $3$-fold $X$ of general type with $p_g(X) \ge 11$. This optimal result naturally leads to more challenging problems, just as the classical Noether inequality does for surfaces, In this paper, we are interested in the following two open questions raised by J. Chen, M. Chen and C. Jiang.

\begin{quesa*} \cite[Question 1.5]{Chen_Chen_Jiang}
	Is there a classification for minimal $3$-folds $X$ of general type satisfying $K_X^3 = \frac{4}{3} p_g(X) - \frac{10}{3}$? Is there any non-Gorenstein minimal $3$-fold $X$ of general type satisfying $K_X^3 = \frac{4}{3} p_g(X) - \frac{10}{3}$?
\end{quesa*}

The answer to the first part of Question A can be viewed as an analogue of E. Horikawa's work \cite{Horikawa1} in dimension three. Regarding the second part, M. Kobayashi \cite{Kobayashi} first constructed a series of examples of minimal $3$-folds $X$ of general type satisfying $K_X^3 = \frac{4}{3} p_g(X) - \frac{10}{3}$. Later, Y. Chen and the first named author \cite{Chen_Hu} generalized Kobayashi's method to obtain more examples. Recently, M. Chen, C. Jiang and B. Li \cite{Chen_Jiang_Li} found two new examples via a different method. Note that all these examples are Gorenstein.

\begin{quesb*} \cite[Question 1.6]{Chen_Chen_Jiang}
	Is there a ``second Noether inequality" in dimension three? Namely, is there a real number $b < \frac{10}{3}$ such that if $K_X^3 > \frac{4}{3} p_g(X) - \frac{10}{3}$ for a minimal $3$-fold $X$ of general type, then $K_X^3 \ge \frac{4}{3} p_g(X) - b$?
\end{quesb*}

For a minimal surface $S$ of general type, since $K_S$ is Cartier, $K_S^2$ is always an integer. Thus the ``second Noether inequality" is naturally that $K_S^2 \ge 2p_g(S) - 3$, and it is actually optimal. However, since the Cartier indices of minimal $3$-folds of general type are unbounded, it is not even clear whether $-\frac{10}{3}$ is an accumulation point of the set
$$
\mathcal{S} = \left\{K_X^3 - \frac{4}{3}p_g(X) \, | \, X \, \mbox{minimal $3$-fold of general type}\right\}.
$$
This is the main obstruction to the existence of the ``second Noether inequality" in dimension three, not to mention the optimal one.

The main purpose of this paper is to answer the above two questions for $3$-folds $X$ of general type with $p_g(X) \ge 11$.
\begin{itemize}
	\item [(1)] We give a complete classification of minimal $3$-folds $X$ of general type with $p_g(X) \ge 11$ satisfying $K^3_X = \frac{4}{3} p_g(X) - \frac{10}{3}$, and we show that they are all Gorenstein.
	
	\item [(2)] We establish the optimal ``second Noether inequality" and even the optimal ``third Noether inequality" for minimal $3$-folds $X$ of general type with $p_g(X) \ge 11$.
\end{itemize}

\subsection{Threefolds on the Noether line}
In this subsection, we introduce our answer to Question A.

\begin{defi} \label{def: Noether line}
	We say a minimal $3$-fold $X$ of general type with $p_g(X) \ge 11$ is \emph{on the (first) Noether line}, if $K_X^3 = \frac{4}{3}p_g(X) - \frac{10}{3}$.
\end{defi}

The first main theorem of this paper gives some basic properties of $3$-folds on the Noether line.

\begin{theorem} \label{thm: main1}
	Let $X$ be a minimal $3$-fold of general type with $p_g(X) \ge 11$ and on the Noether line. Then the following statements hold:
	\begin{itemize}
		\item [(1)] $p_g(X) \equiv 1$ $(\mathrm{mod}$ $3)$, and $X$ is Gorenstein;
		\item [(2)] $h^1(X, \CO_X) = h^2(X, \CO_X) = 0$;
		\item [(3)] The canonical image $\Sigma \subseteq \PP^{p_g(X) - 1}$ of $X$ is a non-degenerate surface of degree $p_g(X) - 2$, and $\Sigma$ is smooth when $p_g(X) \ge 23$;
		\item [(4)] $X$ is simply connected.
	\end{itemize}
\end{theorem}

Let us compare Theorem \ref{thm: main1} with the corresponding results in dimension two. Let $S$ be a minimal surface of general type on the Noether line.
\begin{itemize}
	\item [(1)] E. Bombieri proved that $h^1(S, \CO_S) = 0$ \cite[Lemma 14]{Bombieri}.
	\item [(2)] E. Horikawa proved that the canonical image $\Sigma \subseteq \PP^{p_g(S) - 1}$ of $S$ is a non-degenerate surface of degree $p_g(S) - 2$, and $\Sigma$ is smooth when $p_g(S) \ge 7$ \cite[\S 1]{Horikawa1}.
	\item [(3)] E. Horikawa also proved that $S$ is simply connected \cite[Theorem 3.4]{Horikawa1}.
\end{itemize}
Theorem \ref{thm: main1} (2)--(4) show that similar results hold in dimension three.  Note that there exists a minimal $3$-fold $X$ on the Noether line with $p_g(X) = 19$ whose canonical image is not smooth \cite[Table 10, No. 11]{Chen_Jiang_Li}. Thus the lower bound $p_g(X) \ge 23$ in Theorem \ref{thm: main1} (3) is almost optimal, with the only missing case when $p_g(X) = 22$.

Theorem \ref{thm: main1} (1) answers the second part of Question A in the negative when $p_g(X) \ge 11$. More interestingly, it reveals a novel phenomenon for $3$-folds which does not occur for surfaces. Namely, every minimal $3$-fold $X$ of general type with $p_g(X) \ge 11$ and on the Noether line must satisfy $p_g(X) \equiv 1$ $(\mathrm{mod}$ $3)$. In contrast, for any $m \ge 3$ there exists a minimal surface $S$ of general type on the Noether line with $p_g(S) = m$. This ``congruence modulo $3$" phenomenon suggests that there are two missing Noether inequalities for $3$-folds with $p_g(X) \equiv 2$ and $0$ $(\mathrm{mod}$ $3)$, respectively (which we will establish in Theorem \ref{thm: main3}).

After this paper was finished,  J. Chen informed us that assuming the canonical image being a surface,  he has got a proof of Theorem \ref{thm: main1} (1) using a different method.

Now we turn to the first part of Question A, the classification problem. Let $X$ be a minimal $3$-fold of general type with $p_g(X) \ge 11$ and on the Noether line. By Theorem \ref{thm: main1} (1), we may assume that $p_g(X) = 3m - 2$ for an integer $m \ge 5$. A key observation (see Proposition \ref{prop: fibered minimal model}) is that, there exists a minimal $3$-fold $X_1$ birational to $X$ such that $X_1$ admits a fibration $f: X_1 \to \PP^1$ with general fiber a $(1, 2)$-surface. Here a $(1, 2)$-surface means a smooth surface $S$ of general type with $\vol(S) = 1$ and $p_g(S) = 2$. Moreover, $f$ has a natural section $\Gamma$ which is the horizontal base locus of the relative canonical map of $X_1$ over $\PP^1$ (see \S \ref{subsection: general setting}). Let $X_0$ be the relative canonical model of $X_1$ over $\PP^1$ with the induced fibration $f_0: X_0 \to \PP^1$. Then $\Gamma$ descends to a section $\Gamma_0$ of $f_0$.

\begin{defi} \label{def: relative canonical model}
	We say that $X_0$ is the relative canonical model associated to $X$, and $\Gamma_0$ is the canonical section of $f_0$.
\end{defi}

With the above definition, we now introduce the answer to the first part of Question A.

\begin{theorem} \label{thm: main2}
	Let $X$ be a minimal $3$-fold of general type with $p_g(X) = 3m - 2 \ge 11$ and on the Noether line. Let $X_0$ be the relative canonical model associated to $X$ with the fibration $f_0: X_0 \to \PP^1$. Let $X'_0$ be the blow-up of $X_0$ along the canonical section $\Gamma_0$ of $f_0$. Then the induced fibration $f'_0: X'_0 \to \PP^1$ is factorized as
	$$
	f'_0: X'_0 \stackrel{\rho} \longrightarrow Y \stackrel{q} \longrightarrow \FF_e \stackrel{p} \longrightarrow \PP^1
	$$
	with the following properties:
	\begin{itemize}
		\item [(i)] the Hirzebruch surface $\FF_e$ is isomorphic to $\PP((f_0)_* \omega_{X_0})$;
		\item [(ii)] $q: Y = \PP(\CO_{\FF_e} \oplus \CO_{\FF_e} (-2\bm{s} - (m + e)\bm{l})) \to \FF_e $ is a $\PP^1$-bundle, where $\bm{s}$ is a section on $\FF_e$ with $\bm{s}^2 = -e$ and $\bm{l}$ is a ruling on $\FF_e$;
		\item [(iii)] $\rho: X'_0 \to Y$ is a flat double cover with the branch locus $B = B_1 + B_2$, where $B_1$ is the relative hyperplane section of $Y$, $B_2 \sim 5B_1 + 5(m+e)q^*\bm{l} + 10 q^*\bm{s}$ and $B_1 \cap B_2 = \emptyset$.
	\end{itemize}
	Moreover, if $p_g(X) \ge 23$, then $X_0$ is exactly the canonical model of $X$.
\end{theorem}

In one word, $X_0$ is a divisorial contraction of a double cover of a two-tower of $\PP^1$-bundles over $\PP^1$. Recall that every minimal surface of general type on the Noether line is birationally a double cover of a $\PP^1$-bundle over $\PP^1$ \cite[\S 1]{Horikawa1}. Theorem \ref{thm: main2} asserts that birationally every minimal $3$-fold of general type on the Noether line has a similar structure. A closely related question raised by J. Chen and C. Lai \cite[Question 6.4]{Chen_Lai} asks whether a minimal $n$-fold $X$ of general type ($n \ge 3$) with
$$
K_X^n = \frac{n+1}{n} p_g(X) - \frac{n^2 + 1}{n}
$$
is, under some extra assumptions, birationally a double cover of a tower of $\PP^1$-bundles. Theorem \ref{thm: main2} gives a general affirmative answer to this question for $n = 3$ when $p_g(X) \ge 11$, and no extra assumptions are needed.

\subsection{Two new Noether inequalities} In this subsection, we introduce the second and the third Noether inequalities for $3$-folds of general type.

\begin{theorem} \label{thm: main3}
	Let $X$ be a minimal $3$-fold of general type with $p_g(X) \ge 11$.
	\begin{itemize}
		\item [(1)] Suppose that $K_X^3 > \frac{4}{3}p_g(X) - \frac{10}{3}$. Then we have the optimal inequality
		\begin{equation} \label{eq: main3 second Noether}
			K_X^3 \ge \frac{4}{3}p_g(X) - \frac{19}{6}.
		\end{equation}
		If the equality holds, then $p_g(X) \equiv 2$ $(\mathrm{mod}$ $3)$. Moreover, $X$ has only one non-Gorenstein terminal singularity, and it is of type $\frac{1}{2}(1, -1, 1)$.
		
		\item [(2)] Suppose that $K_X^3 > \frac{4}{3}p_g(X) - \frac{19}{6}$. Then we have the optimal inequality
		\begin{equation} \label{eq: main3 third Noether}
			K_X^3 \ge \frac{4}{3}p_g(X) - 3.
		\end{equation}
		If the equality holds, then $p_g(X) \equiv 0$ $(\mathrm{mod}$ $3)$. Moreover, one of the following two cases occurs:
		\begin{itemize}
			\item [(i)] $X$ has two non-Gorenstein terminal singularities, and they are of the same type $\frac{1}{2}(1, -1, 1)$;
			\item [(ii)] $X$ has only one non-Gorenstein terminal singularity, and it is of type $cA_1/\mu_2$.
		\end{itemize}
	\end{itemize}
    Moreover, if $X$ attains the equality in either \eqref{eq: main3 second Noether} or \eqref{eq: main3 third Noether}, then the statements (2)--(4) in Theorem \ref{thm: main1} also hold for $X$.
\end{theorem}

We remark that the assumption that $p_g(X) \ge 11$ in Theorem \ref{thm: main3} is also optimal, since there exists a minimal $3$-fold of general type with $p_g(X) = 10$ and $K_X^3 = \frac{4}{3}p_g(X) - \frac{33}{10} < \frac{4}{3}p_g(X) - \frac{19}{6}$ \cite[Table 10, No. 10]{Chen_Jiang_Li}.

It is clear that Theorem \ref{thm: main3} (1) answers Question B. More importantly, Theorem \ref{thm: main3} confirms the aforementioned ``congruence modulo $3$" phenomenon that has emerged in Theorem \ref{thm: main1}. It is likely that this phenomenon will also show up in other results regarding the explicit geometry of $3$-folds of general type.

In a forthcoming paper \cite{Hu_Zhang2}, we will establish the complete classification of $3$-folds attaining the equality in \eqref{eq: main3 second Noether} and \eqref{eq: main3 third Noether}, respectively.

\subsection{Idea of the proof} It is clear that to prove the above theorems, we only need to consider the $3$-folds close to the Noether line. To illustrate the main idea of the proof, in the following, we assume that $X$ is a minimal $3$-fold of general type with $p_g(X) \ge 11$ and $K_X^3 < \frac{4}{3}p_g(X) - \frac{8}{3}$. Let
$$
\phi_{K_X}: X \dashrightarrow \PP^{p_g(X) - 1}
$$
be the canonical map of $X$ with the canonical image $\Sigma$.

The first key observation is that $\dim \Sigma = 2$, i.e., $\Sigma$ is a surface. In fact, by a result of Kobayashi (Proposition \ref{prop: Noether dim 3}), we know that $\dim \Sigma \le 2$. The key result we prove here is that $\dim \Sigma \ne 1$ (Proposition \ref{prop: Noether dim 1}). That is to say, there is a gap $\frac{2}{3} = \frac{10}{3} - \frac{8}{3}$ between $3$-folds on the Noether line and $3$-folds whose canonical image is a curve. Interestingly, it is within this gap that we discover two new Noether lines.

Second, using the geometry of the surface $\Sigma$, we prove that up to a birational modification, $X$ admits a fibration 
$$
f: X \to \PP^1
$$ 
with general fiber a minimal $(1, 2)$-surface (Proposition \ref{prop: fibered minimal model}). This important observation opens the door to studying the explicit geometry of $X$. With the help of the fibration $f$, we prove two key estimates in this paper. The first one (Proposition \ref{prop: Noether (1,2)-surface}) is that
$$
K_X^3 \ge p_g(X) + \frac{1}{2}\roundup{\frac{2(p_g(X) - 1)}{3}} - 3.
$$
Because of the ``round-up" appearing in this inequality, three possible residues of $p_g(X)$ modulo $3$ give rise to three different inequalities just as \eqref{eq: Noether 3fold}, \eqref{eq: main3 second Noether} and \eqref{eq: main3 third Noether}. However, to prove that they are indeed the first three Noether inequalities, we need the second key estimate  (Proposition \ref{prop: upper bound of P2}) that
$$
P_2(X) \le \rounddown{2K_X^3} + \rounddown{2K_X^3 - \frac{5(p_g(X) - 1)}{3}} + 7.
$$
Based on the above two estimates and a careful analysis using Reid's Riemann-Roch formula, we manage to prove that \eqref{eq: main3 second Noether} and \eqref{eq: main3 third Noether} are the desired optimal Noether inequalities. Furthermore, when $X$ attains the equality in any of \eqref{eq: Noether 3fold}, \eqref{eq: main3 second Noether} and \eqref{eq: main3 third Noether}, we are able to fully describe the non-Gorenstein singularities on $X$. Thus Theorem \ref{thm: main1} and \ref{thm: main3} are proved.

To prove Theorem \ref{thm: main2}, suppose that $K_X^3 = \frac{4}{3}p_g(X) - \frac{10}{3}$. Let $X_0$ be the relative canonical model of $X$ over $\PP^1$. We first show that $\baselocus|K_{X_0}| = \Gamma_0$ is the canonical section of $f_0: X_0 \to \PP^1$ (Lemma \ref{lem: Noether line KX0}). Denote by $X'_0$ the blow-up of $X_0$ along $\Gamma_0$. Then we show that the rational map 
$$
\varphi: X'_0 \dashrightarrow \PP(f_* \omega_X)
$$ 
is a flat morphism with all fibers integral curves of genus $2$ (Lemma \ref{lem: Noether line flatness}). Thus $X'_0$ is a double cover of a $\PP^1$-bundle over $\PP(f_* \omega_X)$. From this, we are able to fully determine the structure of $X'_0$. When $p_g(X) \ge 23$, we prove that  $K_{X_0}$ is ample (Lemma \ref{lem: Noether line KX0 ample}). Thus $X_0$ is the canonical model of $X$.

\subsection{Notation and conventions} \label{subsection: notation}
In this paper, we adopt the following notation and definitions.

\subsubsection{Varieties and divisors} Let $V$ be a normal variety of dimension $d$. The \emph{geometric genus} $p_g(V)$ and the \emph{second plurigenus} $P_2(V)$ of $V$ are defined as
$$
p_g(V):=h^0(V, K_V), \quad P_2(V) := h^0(V, 2K_V).
$$
For a Weil divisor $L$ on $V$, the \emph{volume} $\vol(L)$ of $L$ is defined as
$$
\vol(L) := \limsup\limits_{n \to \infty} \frac{h^0(V, nL)}{n^d/d!}.
$$
The volume $\vol(K_V)$ is called the \emph{canonical volume} of $V$, and is denoted by $\vol(V)$. We say that $V$ is \emph{minimal}, if $V$ has at worst $\QQ$-factorial terminal singularity and $K_V$ is nef. If $V$ has at worst canonical singularities and $\vol(V) > 0$, we say that $V$ is of general type. If $V$ is of general type and $K_V$ is ample, we say that $V$ is a canonical model. Note that if $V$ is minimal or a canonical model, then $\vol(V) = K_V^d$.

For a linear system $\Lambda$ on $V$, $\movable \Lambda$ and $\baselocus \Lambda$ denote the movable part and the base locus of $\Lambda$, respectively. The rational map 
$$
\phi_{K_V}: V \dashrightarrow \PP^{p_g(V) - 1}
$$
induced by the canonical linear system $|K_V|$ is called the \emph{canonical map} of $V$, and $\phi_{K_V}(V)$ is called the \emph{canonical image} of $V$.

A $\QQ$-divisor on $V$ is \emph{$\QQ$-effective}, if it is $\QQ$-linear equivalent to an effective $\QQ$-divisor.

For an integer $e \ge 0$, $\FF_e$ denotes the Hirzebruch surface $\PP \left(\CO_{\PP^1} \oplus \CO_{\PP^1} (-e)\right)$.


\subsubsection{Birational modifications with respect to divisors} \label{modification} For a normal variety $V$ with $\QQ$-factorial singularities and a Weil (thus $\QQ$-Cartier) divisor $L$ on $V$ with $h^0(V, L) \ge 2$, we can always find a resolution of singularities of $V$
$$
\alpha: V_0 \to V
$$
and a successive blow-ups
$$
\beta \colon V' = V_n \stackrel{\pi_{n-1}}\rightarrow V_{n-1} \rightarrow \cdots \rightarrow V_{i+1} \stackrel{\pi_{i}}\rightarrow V_{i} \rightarrow \cdots \rightarrow V_1 \stackrel{\pi_0}\rightarrow V_0
$$
with the following properties:

\begin{itemize}
	\item[(1)]  $\alpha$ is an isomorphism over the smooth locus of $V$.
	\item [(2)] All $V_i$'s $(i=0, \ldots, n)$ are smooth.
	\item [(3)] Denote $|M_0| = \movable|\rounddown{\alpha^*L}|$. Then each $\pi_i$ is a blow-up along a nonsingular center $W_i$ contained in the base locus of $\mathrm{Mov}|(\pi_0\circ \pi_1 \circ \cdots \circ \pi_{i-1})^*M_0|$.
	\item [(4)] The linear system $\movable|\beta^*M_0|$ is base point free.
\end{itemize}

Set
$$
\pi = \alpha \circ \beta: V' \to V.
$$
This birational modification $\pi$ will be used frequently in this paper.

\subsection*{Acknowledgement} Both authors are grateful to Professors Jungkai Alfred Chen, Meng Chen and Chen Jiang for stimulating questions and fruitful discussions.

Y.H. is sponsored by Shanghai Pujiang Program Grant No. 21PJ1405200. T.Z. is supported by the National Natural Science Foundation of China (NSFC) General Grant No. 12071139 and the Science and Technology Commission of Shanghai Municipality (STCSM) Grant No. 18dz2271000.


\section{Fibered minimal model for some threefolds with small volume}  \label{section: fibered model}

Let $X$ be a minimal $3$-fold of general type with $p_g(X) \ge 2$. Let
$$
\phi_{K_X}: X \dashrightarrow \PP^{p_g(X) - 1}
$$
be the canonical map of $X$ with the canonical image $\Sigma$. Denote by
$\pi: X' \to X$ the birational modification of $X$ with respect to $K_X$ as in \S \ref{modification}. We have the following commutative diagram
$$
\xymatrix{
	X' \ar[d]_{\pi} \ar[rr]^{\psi} \ar[drr]^{\phi_{M}} & &  \Sigma' \ar[d]^{\tau}  \\
	X  \ar@{-->}[rr]_{\phi_{K_X}} & & \Sigma
}
$$
Here $\phi_M$ is the morphism induced by the linear system $|M| = \movable|\rounddown{\pi^*K_X}|$, and $X' \stackrel{\psi}\rightarrow \Sigma' \stackrel{\tau} \rightarrow \Sigma$ is the Stein factorization of $\phi_M$. We may write
$$
\pi^*K_X = M + Z,
$$
where $Z \ge 0$ is a $\QQ$-divisor.

The main result in this section is the following.
\begin{prop} \label{prop: fibered minimal model}
	With the above notation, suppose that $p_g(X) \ge 7$, $\dim \Sigma = 2$ and $K_X^3 < \frac{4}{3}p_g(X) - 2$. Then $\deg \Sigma = p_g(X) - 2$, and $\tau$ is an isomorphism. Moreover, there exists a minimal $3$-fold $X_1$ birational to $X$ such that $X_1$ admits a fibration $f: X_1 \to \PP^1$ with general fiber $F_1$ a $(1,2)$-surface.
\end{prop}

\begin{proof}
    Since $\Sigma \subseteq \PP^{p_g(X)-1}$ is non-degenerate, we have $\deg \Sigma \ge p_g(X) - 2$. In the following, we prove the proposition by steps.
	
	\textbf{Step 1}. Denote by $C$ a general fiber of $\psi$. In this step, we show that
	$$
	((\pi^*K_X) \cdot C) \ge  1.
	$$
	
	Let $S \in |M|$ be a general member. By Bertini's theorem, $S$ is smooth, and we have
	$$
	M|_S \equiv dC,
	$$
	where $d:=(\deg \tau) \cdot (\deg \Sigma) \ge p_g(X)-2$. Since $p_g(X) \ge 7$, we have
	$$
	K_X^3 < \frac{4}{3} p_g(X) - 2 < 2p_g(X) - 4.
	$$
	By \cite[Theorem 4.1]{Chen_Chen_Jiang}, we deduce that $g(C) = 2$. Let $\sigma: S \to S_0$ be the contraction onto the minimal model of $S$, and let $C_0 = \sigma_* C$. Since
	$$
	K_S=(K_{X'}+S)|_S\ge 2 M|_S \equiv 2dC \ge 2(p_g(X) - 2)C \ge 10C,
	$$
	we deduce that $K_{S_0} - 10C_0$ is pseudo-effective. In particular,
	$$
	K_{S_0}^2 \ge 10(K_{S_0} \cdot C_0) \ge 10.
	$$
	Since $g(C)=2$, we have
	$$
	(K_{S_0} \cdot C_0) = \left((\sigma^*K_{S_0}) \cdot C\right) \le (K_S \cdot C) = 2.
	$$
	By the Hodge index theorem, we deduce that $C_0^2 = 0$. This implies that $$
	\left((\sigma^*K_{S_0}) \cdot C\right) = (K_{S_0} \cdot C_0) = 2.
	$$
	By \cite[Corollary 2.3]{Chen_Chen_Jiang}, $2(\pi^*K_X)|_S - \sigma^*K_{S_0}$ is $\QQ$-effective. Thus
	$$
	\left( (\pi^*K_X) \cdot C \right) \ge \frac{1}{2} \left((\sigma^*K_{S_0}) \cdot C \right) = 1.
	$$
	
	\textbf{Step 2}. In this step, we prove that $\deg \Sigma = p_g(X) - 2$ and that $\tau$ is an isomorphism. Moreover, we construct a relatively minimal fibration from a birational model of $X$ to $\PP^1$.
	
	By the argument in the proof of \cite[Theorem 4.2]{Chen_Chen_Jiang} and the assumption, we have
	$$
	\frac{4}{3}p_g(X) - 2 > K_X^3 \ge \left((\pi^*K_X)|_S\right)^2 \ge \frac{2}{3} (2d - 1).
	$$
	Thus $d = p_g(X) - 2$, which implies that
	$$
	\deg \tau = 1, \quad \deg \Sigma = p_g(X) - 2.
	$$
	By \cite[\S 10]{Nagata}, there is a Hirzebruch surface $\mathbb{F}_e$ for some $e \ge 0$ and a morphism
	$$
	r: \mathbb{F}_e \to  \PP^{p_g(X) - 1}
	$$
	induced by the linear system $|\bm{s} + (e + k)\bm{l}|$ such that $\Sigma = r(\FF_e)$. Here $\bm{l}$ is a ruling of the natural fibration $p: \FF_e \to \PP^1$, $\bm{s}$ is a section of $p$ with $\bm{s}^2 = -e$, and $k \in \ZZ_{\ge 0}$ such that $\deg \Sigma = e + 2k$. In particular, $\Sigma$ is normal. Thus $\tau$ is an isomorphism.
	
	Replacing $X'$ by its birational modification, we may assume that there is a surjective morphism $\varphi: X' \to \FF_e$ such that $\psi = r \circ \varphi$. Thus we obtain a fibration
	$$
	f':= p \circ \varphi: X' \to \FF_e \to \PP^1
	$$
	with a general fiber $F' = \varphi^*\bm{l}$. Let $\zeta: X' \dashrightarrow X_1$ be the contraction of $X'$ onto its relative minimal model $X_1$ over $\PP^1$. Up to a birational modification, we may assume that $\zeta$ is a morphism. Then we obtain a relatively minimal fibration
	$$
	f_1: X_1 \to \PP^1
	$$
	with a general fiber $F_1$. Here $\mu:= \zeta|_{F'}: F' \to F_1$ is just the contraction onto the minimal model of $F'$.
	
	\textbf{Step 3}. In this step, we prove that $F_1$ is a $(1, 2)$-surface.
	
	Since $\dim \Sigma = 2$, the natural restriction map
	$$
	H^0(X', K_{X'}) \to H^0(F', K_{F'})
	$$
	has the image of dimension at least two. In particular, $p_g(F_1) = p_g(F') \ge 2$. To show that $F_1$ is a $(1, 2)$-surface, it suffices to show that $K_{F_1}^2 = 1$.
	
	From \textbf{Step 2} and the assumption that $p_g(X) \ge 7$, we deduce that $e+k \ge \frac{1}{2} \deg \Sigma = \frac{1}{2} p_g(X) - 1 \ge \frac{5}{2}$, i.e., $e + k \ge 3$. Also recall in \textbf{Step 2} that $M = \varphi^*(\bm{s} + (e + k) \bm{l})$. Thus $\pi^*K_X - (e+k)F' \ge 0$. By \cite[Corollary 2.3]{Chen_Chen_Jiang}, $(1 + \frac{1}{e+k})(\pi^*K_X)|_{F'} - \mu^*K_{F_1}$ is $\QQ$-effective.  In particular, $\frac{4}{3} (\pi^*K_X)|_{F'} - \mu^*K_{F_1}$ is $\QQ$-effective. On the other hand, by the assumption, we have
	$$
	\frac{4}{3} p_g(X) - 2> K_X^3 \ge d \left((\pi^*K_X) \cdot C \right) = (p_g(X) - 2) \left((\pi^*K_X) \cdot C \right).
	$$
	It follows from \textbf{Step 1} that
	$$
	\left((\mu^*K_{F_1}) \cdot C \right) \le \frac{4}{3} \left((\pi^*K_X)|_{F'} \cdot C \right) < \frac{4}{3}\cdot\frac{\frac{4}{3}p_g(X) - 2}{p_g(X) - 2} < 2,
	$$
	where the last inequality follows from the assumption that $p_g(X) \ge 7$. Let $C_1 = \mu_*C$. Then the above inequality means that  $(K_{F_1} \cdot C_1) < 2$.  Thus $(K_{F_1} \cdot C_1) = 1$. Note that $K_{F'} \ge (\pi^*K_X) |_{F'} \ge M|_{F'} \ge C$. By the Hodge index theorem and the parity, we deduce that $K_{F_1}^2 = C_1^2 = 1$ and
	$$
	K_{F_1} = \mu_*((\pi^*K_X) |_{F'} ) = C_1.
	$$

	\textbf{Step 4}. In this step, we show that $X_1$ is minimal. By \cite[Lemma 3.2]{Chen_Chen_Jiang}, it suffices to show that
	$$
	\mu^*K_{F_1} = (\pi^*K_X)|_{F'}.
	$$
	
	By the conclusion in \textbf{Step 3}, we may write
	$$
	\mu^*K_{F_1} = (\pi^*K_X) |_{F'} + A_1 - A_2,
	$$
	where $A_1$ and $A_2$ are both effective $\QQ$-divisors on $F'$ supported on the exceptional locus of $\mu$ with no common irreducible components. Thus we have
	$$
	0 =  \left((\mu^*K_{F_1}) \cdot A_2 \right) = \left( \left( (\pi^*K_X) |_{F'} + A_1 - A_2 \right) \cdot A_2  \right) \ge - A_2^2 \ge 0,
	$$
	i.e., $A_2^2 = 0$. This implies that $A_2 = 0$, and thus $\mu^*K_{F_1} \ge (\pi^*K_X)|_{F'}$. As a result, we have
	$$
	1 = (\mu^*K_{F_1})^2 \ge \left( (\pi^*K_X)|_{F'} \right)^2 \ge ((\pi^*K_X) \cdot C) \ge 1,
	$$
	which forces $(\mu^*K_{F_1})^2 = ( (\pi^*K_X)|_{F'} )^2$. By the Hodge index theorem again, we deduce that $\mu^*K_{F_1} = (\pi^*K_X)|_{F'}$. Thus the proof is completed.
\end{proof}


\section{Geometry of $3$-folds fibered by $(1, 2)$-surfaces} \label{section: geometry}

Throughout this section, let $X$ be a minimal $3$-fold of general type with a fibration
$$
f: X \to \PP^1
$$
such that the general fiber $F$ of $f$ is a $(1, 2)$-surface. We always assume that the canonical image of $X$ is a surface. In particular, $p_g(X) \ge 3$.

\subsection{General Setting} \label{subsection: general setting}
Since $X$ is minimal, by the adjunction, $F$ is also minimal. Thus $|K_F|$ has a unique base point. Consider the relative canonical map $X \dashrightarrow \PP(f_* \omega_X)$ of $X$ over $\PP^1$. It is clear that the indeterminacy locus of this rational map contains a section $\Gamma$ of $f$ such that $\Gamma \cap F$ is just the base point of $|K_F|$.

Following the notation in \S \ref{section: fibered model}, taking a birational modification $\pi: X' \to X$ as in \S \ref{modification} with respect to $K_X$, we may write
\begin{equation} \label{eq: modification}
	\pi^*K_X = M + Z,
\end{equation}
where $|M| = \movable|\rounddown{\pi^*K_X} |$ is base point free and $Z \ge 0$ is a $\QQ$-divisor. We have a similar commutative diagram
$$
\xymatrix{
	& & X' \ar[d]_{\pi} \ar[lld]_{f'} \ar[rr]^{\psi} \ar[drr]^{\phi_{M}} & &  \Sigma' \ar[d]^{\tau}  \\
	\mathbb{P}^1 & & X  \ar@{-->}[rr]_{\phi_{K_X}} \ar[ll]^f  & & \Sigma
}
$$
as in \S \ref{section: fibered model}, where $\phi_{K_X}$ is the canonical map of $X$, $\phi_M$ is the morphism induced by $|M|$, $X' \stackrel {\psi}\rightarrow \Sigma' \stackrel{\tau} \rightarrow \Sigma$ is the Stein factorization of $\phi_M$, and $f' = f \circ \pi$ is the induced fibration. Denote by $F'$ a general fiber of $f'$.

Note that $X$ has at worst terminal singularities. We may write
\begin{equation} \label{eq: exceptional divisor}
	K_{X'} = \pi^*K_X + E_{\pi},
\end{equation}
where $E_{\pi} \ge 0$ is a $\pi$-exceptional $\QQ$-divisor.

Let $C$ be a general fiber of $\psi$. Since $\dim \Sigma = 2$, $C$ is a curve, and we have
$$
M|_S \equiv dC,
$$
where $d = (\deg \tau) \cdot (\deg \Sigma) \ge p_g(X) - 2 \ge 1$.

\begin{remark} \label{rmk: d=pg-2}
	By Proposition \ref{prop: fibered minimal model}, if $p_g(X) \ge 7$ and $K_X^3 < \frac{4}{3}p_g(X) - 2$, then $\deg \tau = 1$ and $\deg \Sigma = p_g(X) - 2$. In particular, $d = p_g(X) - 2$.
\end{remark}

Recall the following two lemmas from \cite{Hu_Zhang}.

\begin{lemma}{\cite[Lemma 2.2]{Hu_Zhang}} \label{lem: g2}
	We have $|M||_{F'} = \movable |(\pi|_{F'})^*K_F|$, where $\pi|_{F'}: F' \to F$ coincides with the blow-up of the unique base point of $|K_F|$. In particular, $g(C)=2$ and there is a rational map $p: \Sigma' \dashrightarrow \PP^1$ such that $f' = p \circ \psi$.
\end{lemma}

\begin{lemma}{\cite[Lemma 2.4]{Hu_Zhang}} \label{lem: E0}
	There exists a unique $\pi$-exceptional prime divisor $E_0$ on $X'$ such that
	\begin{itemize}
		\item [(1)] $\coeff_{E_0} (Z) = \coeff_{E_0}(E_{\pi}) = 1$;
		\item [(2)] $\pi(E_0)=\Gamma$, $\phi_M (E_0) = \Sigma$;
		\item [(3)] $(E_0 \cdot C) = (Z \cdot C) = (E_{\pi} \cdot C) = ((\pi^*K_X) \cdot C) = 1$.
	\end{itemize}
\end{lemma}

\begin{lemma} \label{lem: nef Hodge bundle}
	The sheaf $f_* \omega_X$ is locally free of rank two over $\PP^1$, and it is nef. Moreover, $f_* \omega_X$ is not ample if and only if $\Sigma$ is a cone over a smooth rational curve of degree $p_g(X) - 2$.
\end{lemma}

\begin{proof}
	It is clear that $f_* \omega_X$ is torsion free over $\PP^1$ of rank $p_g(F) = 2$. Thus it is locally free. We may write
	$$
	f_* \omega_X = \CO_{\PP^1}(a) \oplus \CO_{\PP^1}(b)
	$$
	with $a, b \in \ZZ$ and $a \ge b$. Since the canonical image $\Sigma$ is a surface, we deduce that $b \ge 0$. Otherwise, $\Sigma \simeq \PP^1$. Therefore, $a \ge b \ge 0$, and $f_* \omega_X$ is nef.
	
	Consider the relative canonical map
	$$
	\phi: X' \dashrightarrow V:=\PP(f'_* \omega_{X'})
	$$
	of $X'$ over $\PP^1$. It is known that $\phi$ is induced by the linear system $|K_{X'} + tF'|$ for a sufficiently large integer $t$. Since $f'_*$ is left exact, we know that $f'_* \CO_{X'}(M)$ is a subsheaf of  $f'_* \omega_{X'}$. Note that $f'_*\CO_{X'}(M)$ is also a direct sum of two line bundles over $\PP^1$ and $h^0(X', M) = h^0(X', K_{X'})$. It follows that $f'_* \CO_{X'}(M) = f'_* \omega_{X'}$. By the projection formula, we have $f'_*\CO_{X'}(M + tF') = f'_*\CO_{X'}(K_{X'} + tF')$. We deduce that $h^0(X', M + tF') = h^0(X', K_{X'} + tF')$. In particular, $\movable|K_{X'} + tF'| = |M + tF'|$ is base point free. Thus $\phi$ is a morphism. Since $\phi$ is induced by the natural morphism 
	$f'^*f'_*\omega_{X'}\to \omega_{X'}$. We have $K_{X'}\ge\phi^*H$, where $H$ is a relative hyperplane section of $V$. As $f_*\omega_X$ is nef over $\PP^1$, $f_*\omega_X$ is generated by global sections. Since $X$ is terminal, we have $f'_* \omega_{X'} = f_* \omega_X$. It follows that $|H|$ is base point free  and $h^0(V, H) = h^0(\PP^1, f_* \omega_X) = p_g(X) = h^0(X', M)$. In particular, we have
	$$
	|M| = \phi^*|H|.
	$$ 
	As a result, $\Sigma$ is just the image of $V$ under the morphism induced by $|H|$. Since $V$ is a Hirzebruch surface, we know that $\Sigma$ is a cone if and only if $H$ is not ample, which is equivalent to the non-ampleness of $f_* \omega_X$. Thus the proof is completed.
\end{proof}

\begin{lemma}\label{lem: H1H2}
	We have $h^1(X, \CO_X) = h^2(X, \CO_X)=0$. In particular,  $\chi(\omega_X)=p_g(X)-1$.
\end{lemma}
\begin{proof}
	By \cite[Theorem 2.1]{Kollar_2}, $R^1f_*\omega_X$ and $R^2f_*\omega_X$ are both torsion free sheaves. Now $q(F)=0$. We deduce that $R^1f_* \omega_X = 0$. Moreover, $R^2f_*\omega_X = \omega_{\PP^1}$. Thus
	$$
	h^1(X, \CO_X) =h^2(X,\omega_X) =h^0(\PP^1, \omega_{\PP^1}) + h^1(\PP^1, R^1f_*\omega_X)=0.
	$$
	By Lemma \ref{lem: nef Hodge bundle}, $h^1(\PP^1, f_* \omega_X) = 0$. It follows that
	$$
	h^2(X, \CO_X) = h^1(X, \omega_X) = h^1(\PP^1, f_*\omega_X) + h^0(\PP^1, R^1f_*\omega_X) = 0.
	$$
	Finally, we conclude that
	$$
	\chi(\omega_X) = p_g(X) + h^1(X, \CO_X) - h^2(X, \CO_X) - 1 = p_g(X)-1.
	$$
	The proof is completed.	
\end{proof}

\subsection{Two key estimates} \label{subsection: Noether (1,2)-surface}

In this subsection, we prove two important estimates.

\begin{prop} \label{prop: Noether (1,2)-surface}
	The following inequalities hold:
	\begin{itemize}
		\item [(1)] $\displaystyle{(K_X \cdot \Gamma) \ge \frac{1}{2}  \roundup{\frac{2(d - 2)}{3}}}$;
		\item [(2)] $\displaystyle{K_X^3 \ge d + \frac{1}{2} \roundup{\frac{2(d - 2)}{3}}}$.
	\end{itemize}
    Moreover, if the equality in (2) holds, so does the equality in (1).
\end{prop}

\begin{proof}
	Let $E_0$ be the unique $\pi$-exceptional prime divisor as in Lemma \ref{lem: E0}. Take a general member $S \in |M|$. By Bertini's theorem and Lemma \ref{lem: E0} (2), $S$ is smooth and $S|_{E_0}$ is irreducible. By Lemma \ref{lem: g2}, $S \cap F' = C$. Now a general fiber of $f'|_S: S \to \PP^1$ is $S \cap F' = C$, and it is just a general fiber of $\psi|_S$. We conclude that $f'|_S$ and $\psi|_S$ are identical to each other.  
	
	Denote $\Gamma_S = E_0|_S$. By Lemma \ref{lem: E0} (2) and (3), $\Gamma_S$ is a section of $\psi|_S$. Therefore, by Lemma \ref{lem: E0} (1), we may write
	\begin{equation} \label{eq: EVZV}
		E_{\pi}|_S = \Gamma_S + E_V, \quad Z|_S = \Gamma_S + Z_V.
	\end{equation}
	Here $E_V \ge 0$ and $Z_V \ge 0$ are $\QQ$-divisors on $S$. By Lemma \ref{lem: E0} (3),
	$$
	(E_V \cdot C) = \left((E_{\pi} - E_0) \cdot C\right) = 0, \quad (Z_V \cdot C) = \left((Z-E_0) \cdot C\right)=0.
	$$
	We deduce that both $E_V$ and $Z_V$ are vertical with respect to $\psi|_S$.
	
	By the adjunction formula, \eqref{eq: modification}, \eqref{eq: exceptional divisor} and \eqref{eq: EVZV},
	\begin{equation}\label{eq: K_S}
		K_S = (K_{X'} + S)|_S \equiv 2dC + 2 \Gamma_S + E_V + Z_V.
	\end{equation}
	Denote by $\sigma: S \to S_0$ the contraction onto its minimal model $S_0$. By \cite[Corollary 2.3]{Chen_Chen_Jiang}, we have
	\begin{equation} \label{eq: restriction comparison}
		\left. \left(\pi^*K_X\right) \right|_S \sim_{\QQ} \frac{1}{2} \sigma^* K_{S_0} + H_S,
	\end{equation}
	where $H_S \ge 0$ is a $\QQ$-divisor on $S$. Therefore, by Lemma \ref{lem: E0} (3), $((\sigma^*K_{S_0}) \cdot C) \le 2((\pi^*K_X) \cdot C) = 2$. Let $C_0 = {\sigma}_*C$ and $\Gamma_{S_0} = {\sigma}_*\Gamma_S$. Then we have $(K_{S_0} \cdot C_0) \le 2$. On the other hand, by \eqref{eq: K_S}, we have
	\begin{align}\label{eq: K_S0}
		K_{S_0} \equiv 2dC_0 + 2\Gamma_{S_0} + {\sigma}_*(E_V + Z_V).
	\end{align}
	We deduce that $(K_{S_0} \cdot C_0) \ge 2d C_0^2 \ge 2 C_0^2$. By parity, it follows that $(K_{S_0} \cdot C_0) = 2$ and $C_0^2 = 0$. In particular, the fibration $\psi|_S$ descends to a fibration $S_0 \to \mathbb{P}^1$ whose general fiber is $C_0$ with $g(C_0)=2$. Moreover,
	$$
	(H_S \cdot C) = \left( \left(\pi^*K_X \right) \cdot C \right)-\frac{1}{2}(K_{S_0} \cdot C_0) = 0.
	$$
	This implies that $H_S$ is vertical with respect to $\psi|_S$.
	
	Since $\Gamma_{S_0}$ is a section of the fibration $S_0 \to \mathbb{P}^1$, we have $g(\Gamma_{S_0}) = 0$. By the adjunction formula on $S_0$ and \eqref{eq: K_S0}, we have
	\begin{align}
		-2 & =(K_{S_0} \cdot \Gamma_{S_0}) + \Gamma_{S_0}^2 \nonumber \\
		& = (K_{S_0} \cdot \Gamma_{S_0}) + \frac{1}{2} \left( \left(K_{S_0} - 2d  C_0 - {\sigma}_*(E_V + Z_V) \right) \cdot \Gamma_{S_0}\right) \label{eq: remark 1} \\
		& =\frac{3}{2} (K_{S_0} \cdot \Gamma_{S_0}) - d  - \frac{1}{2} \left({\sigma}_*(E_V+Z_V) \cdot \Gamma_{S_0} \right) . \nonumber
	\end{align}
	We deduce that
	\begin{equation}\label{eq: K_S0Gamma_S0}
		(K_{S_0} \cdot \Gamma_{S_0}) =\frac{2}{3}d-\frac{4}{3}+\frac{1}{3}\left({\sigma}_*(E_V+Z_V) \cdot \Gamma_{S_0} \right)
		\ge \roundup{\frac{2 (d-2)}{3}},
	\end{equation}
	where the last inequality holds because ${\sigma}_*(E_V + Z_V)$ is vertical with respect to the fibration $S_0 \to \mathbb{P}^1$ and $(K_{S_0}\cdot\Gamma_{S_0})$ is an integer. Together with \eqref{eq: restriction comparison} and the fact that $(H_S \cdot \Gamma_S) \ge 0$, we deduce that
	$$
	(K_X \cdot \Gamma)  = \left((\pi^*K_X)|_S \cdot \Gamma_S \right) \ge \frac{1}{2} \left((\sigma^* K_{S_0}) \cdot \Gamma_S \right)
	\ge \frac{1}{2} \roundup{\frac{2 (d-2)}{3}}
	$$
	Thus the inequality (1) is proved.
	
	For (2), note that
	$$
	K_X^3   \ge ((\pi^* K_X) \cdot M^2) + \left( (\pi^*K_X)|_S \cdot Z|_S \right).
	$$
	By Lemma \ref{lem: E0} (3),
	\begin{equation} \label{eq: K_XM^2}
		\left((\pi^*K_X) \cdot M^2 \right) = d \left((\pi^*K_X) \cdot C \right) = d.
	\end{equation}
	By (1), we have
	\begin{equation} \label{eq: K_XMZ}
		\left((\pi^*K_X)|_S \cdot Z|_S \right) \ge \left((\pi^*K_X)|_S \cdot \Gamma_S \right) = (K_X \cdot \Gamma) \ge \frac{1}{2} \roundup{\frac{2 (d-2)}{3}}.
	\end{equation}
	Combine the above inequalities together, and we deduce that
	$$
	K_X^3 \ge  d + \frac{1}{2} \roundup{\frac{2 (d-2)}{3}}.
	$$
	This proves (2).
\end{proof}

\begin{prop}\label{prop: Noether equality (1,2)-surface}
	Keep the same notation as in the proof of Proposition \ref{prop: Noether (1,2)-surface}. Suppose that the equality in Proposition \ref{prop: Noether (1,2)-surface} (2) holds. Then we have the following equalities:
	\begin{itemize}
		\item [(1)] $\displaystyle K_{S_0}^2 = 4d + 2 \roundup{\frac{2 (d - 2)}{3}}$;
		\item [(2)] $(K_{S_0} \cdot \sigma_*(E_V+Z_V)) = 0$;
		\item [(3)] $\sigma^*K_{S_0} \sim_{\QQ} 2(\pi^*K_X)|_S$.
	\end{itemize}
	Moreover, if $d \equiv 2$ $(\mathrm{mod}$ $3)$, then $\sigma_*(E_V+Z_V)=0$ and $K_{S_0} \equiv 2dC_0 + 2\Gamma_{S_0}$.
		
		
\end{prop}

\begin{proof}
	By \eqref{eq: K_S0} and \eqref{eq: K_S0Gamma_S0}, we have
	\begin{align*}
		K_{S_0}^2 & = 2d(K_{S_0} \cdot C_0) + 2(K_{S_0} \cdot \Gamma_{S_0}) + (K_{S_0} \cdot \sigma_*(E_V+Z_V)) \\
		& \ge 4d + 2\roundup{\frac{2 (d-2)}{3}} + \left(K_{S_0}\cdot\sigma_*(E_V+Z_V)\right).
	\end{align*}
	Together with \eqref{eq: restriction comparison}, we deduce that
	$$
	K_X^3 \ge ((\pi^*K_X)|_S)^2\ge \frac{1}{4}K_{S_0}^2 \ge d+\frac{1}{2} \roundup{\frac{2 (d-2)}{3}} + \frac{1}{4}\left(K_{S_0}\cdot\sigma_*(E_V+Z_V)\right).
	$$
	Now by our assumption, $K_X^3 = d+\frac{1}{2} \roundup{\frac{2 (d-2)}{3}}$. We conclude that all the above inequalities become equalities. In particular, we have $K_{S_0}^2 = 4K_X^3$ and $\left(K_{S_0}\cdot\sigma_*(E_V+Z_V)\right) = 0$. Thus the equalities (1) and (2) hold. Moreover, we have $4((\pi^*K_X)|_S)^2 = 2((\pi^*K_X)|_S\cdot\sigma^*K_{S_0}) =K_{S_0}^2$. 
	
	By the Hodge index theorem, \eqref{eq: restriction comparison} and the above equalities, we deduce that $H_S \equiv 0$. Since $H_S$ is effective, we have $H_S=0$. Thus
	$$
	\sigma^*K_{S_0}\sim_{\QQ} 2(\pi^*K_X)|_S.
	$$
	
	Suppose that $d \equiv 2$ $(\mathrm{mod}$ $3)$. Then $\frac{2d}{3} - \frac{4}{3} = \roundup{\frac{2 (d-2)}{3}}$. Note that \eqref{eq: K_S0Gamma_S0} now becomes an equality. Thus
	$$
	\left({\sigma}_*(E_V+Z_V) \cdot \Gamma_{S_0} \right) = 0
	$$
	Moreover, by \eqref{eq: K_S0} and the equality (2), we conclude that
	$$
	\left({\sigma}_*(E_V + Z_V) \right)^2 =  - \left({\sigma}_*(E_V + Z_V)  \cdot (2\Gamma_{S_0} + 2d C_0) \right) = 0.
	$$
	Thus $\sigma_*(E_V+Z_V)=0$. By \eqref{eq: K_S0}, we have $K_{S_0} \equiv 2dC_0 + 2\Gamma_{S_0}$. The whole proof is completed.
\end{proof}

\begin{prop}\label{prop: upper bound of P2}
	One of the following inequalities holds:
	\begin{itemize}
		\item [(1)] $\displaystyle{P_2(X) \le \rounddown{2K_X^3} + \rounddown{2K_X^3 - \frac{5(p_g(X)-1)}{3}} + 7}$;
		
		\item [(2)] $\displaystyle{K_X^3 \ge \frac{4}{3}p_g(X) - \frac{17}{6}}$.
	\end{itemize}
\end{prop}

\begin{proof}
	Set $|M_0|=\mathrm{Mov}|2K_{X'}|$, $|M_1|=\mathrm{Mov}|2K_{X'}-M|$ and $|M_2|=\mathrm{Mov}|2K_{X'}-2M|$. Replacing $X'$ by a further blow-up, we may assume that $|M_0|$, $|M_1|$ and $|M_2|$ are all base point free. By Bertini's theorem, we may take a smooth general member $S\in|M|$. It is easy to see that
	\begin{equation} \label{eq: decomp P_2}
		P_2(X)=u_0+u_1+h^0(X', 2K_{X'}-2M),
	\end{equation}
	where
	$$
	u_i = \dim \mathrm{Im} \left(H^0(X', M_i) \to H^0(S, M_i|_S)\right) \quad (i=0, 1).
	$$
	Note that we always have
	\begin{equation}\label{eq: step 0}
		K_X^3\ge((\pi^*K_X)|_S)^2.
	\end{equation}
	
	\textbf{Step 1}. In this step, we prove that
	\begin{equation} \label{eq: step 1}
		u_0 \le \rounddown{2K_X^3} + 2.
	\end{equation}
	
	Consider the complete linear system $|M_0|_S|$. By our assumption, it induces a morphism $\phi_0: S \to \PP^{h^0(S, M_0|_S) - 1}$. If $\dim \phi_0(S) = 2$, by \cite[Lemma 1.8]{Ohno},
	$$
	4 \left((\pi^*K_X)|_S\right)^2 \ge (M_0|_S)^2 \ge 2h^0(S, M_0|_S) - 4.
	$$
	If $\dim \phi_0(S) = 1$, since $M_0|_S \ge M|_S$, the general fiber of $\phi_0$ is identical to $C$. Now $M_0|_S \equiv b C$, where $b \ge h^0(S, M_0|_S) - 1$. Thus by Lemma \ref{lem: E0} (3),
	$$
	2 \left((\pi^*K_X)|_S \right)^2 \ge b \left((\pi^*K_X) \cdot C \right) \ge  h^0(S, M_0|_S) - 1.
	$$
	Thus in both cases, we always have
	$$
	2 \left((\pi^*K_X)|_S \right)^2 \ge h^0(S, M_0|_S) - 2.
	$$
	Together with \eqref{eq: step 0}, it follows that
	$$
	u_0 \le h^0(S, M_0|_S) \le 2 \left((\pi^*K_X)|_S \right)^2 + 2\le 2K_X^3+2.
	$$
	Since $u_0$ is an integer, we deduce that $u_0\le \rounddown{2K_X^3}+2$.
	
	\textbf{Step 2}. In this step, we prove that
	\begin{equation} \label{eq: step 2}
		u_1 \le \rounddown{2K_X^3 - \frac{5(p_g(X) - 1)}{3}} + 4.
	\end{equation}
	
	Similarly as in \textbf{Step 1}, the complete linear system $|M_1|_S|$ induces a morphism $\phi_1: S \to \PP^{h^0(S, M_1|_S) - 1}$. If $\dim \phi_1(S) = 2$, by \cite[Lemma 1.8]{Ohno}, $(M_1|_S)^2 \ge 2h^0(S, M_1|_S) - 4$. It follows that
	$$
	4\left((\pi^*K_X)|_S\right)^2 \ge (M_1|_S + M|_S)^2 \ge 2h^0(S, M_1|_S) - 4 + 2(M_1|_S \cdot M|_S).
	$$
	Since $\phi_1$ does not contract $C$, the linear system $|M_1|_S||_C$ induces a finite morphism on $C$. Note that $g(C) = 2$ by Lemma \ref{lem: g2}. We deduce that $(M_1|_S \cdot C) \ge 2$.  Recall that  we have $M|_S \equiv dC $, where $d\ge p_g(X)-2$. Thus
	$$
	(M_1|_S \cdot M|_S) \ge d (M_1|_S \cdot C) \ge 2d \ge 2p_g(X)-4 .
	$$
	Combine the above inequalities with \eqref{eq: step 0} together. We deduce that
	$$
	h^0(S, M_1|_S) \le 2 \left((\pi^*K_X)|_S \right)^2 + 2 - (2p_g(X) - 4) \le 2K_X^3+6-2p_g(X).
	$$
	Since $p_g(X) \ge 3$, we deduce that
	$$
	u_1 \le h^0(S, M_1|_S) \le 2K_X^3+\frac{17}{3}-\frac{5}{3}p_g(X) = 2K_X^3 - \frac{5(p_g(X) - 1)}{3} + 4.
	$$
	Since $u_1$ is an integer, we know that \eqref{eq: step 2} holds in this case.
	
	If $\dim \phi_1(S) = 1$, since $M_1|_S \ge M|_S$, the general fiber of $\phi_1$ is just $C$. Thus $M_1|_S \equiv bC$, where $b \ge h^0(S, M_1|_S) - 1$. Note that $M|_S \equiv d C$ as before. Therefore, the divisor $2(\pi^*K_X)|_S - (b +d)C$ on $S$ is pseudo-effective. Let $\sigma: S \to S_0$ be the contraction morphism onto the minimal model of $S$.  As has been proved in Proposition \ref{prop: Noether (1,2)-surface}, the fibration $\psi|_S: S\to \mathbb{P}^1$ descends to a fibration $S_0\to\mathbb{P}^1$ with a general fiber $C_0 = \sigma_*C$. In the meantime, by \eqref{eq: K_S0}, we may write
	$$
	K_{S_0} \equiv 2dC_0 + 2\Gamma_{S_0}+\Delta_0,
	$$ where $\Gamma_{S_0}$ is a section of the fibration $S_0 \to \PP^1$ and $\Delta_0$ is an effective divisor which is vertical with respect to the fibration $S_0\to\mathbb{P}^1$. By the adjunction formula, we have
	$$
	-2=(K_{S_0}\cdot\Gamma_{S_0})+\Gamma_{S_0}^2=2d+3\Gamma_{S_0}^2+(\Delta_0\cdot\Gamma_{S_0}).
	$$
	We deduce that
	$$
	\Gamma_{S_0}^2=-\frac{2}{3}d-\frac{2}{3}-\frac{(\Gamma_{S_0}\cdot\Delta_0)}{3}.
	$$
	This implies that the divisor
	$$
	K_{S_0} - \frac{2(d-2)}{3} C_0 \equiv \frac{2}{3}
	\left( \left( 2d+2 \right) C_0 + 3\Gamma_{S_0}+\frac{3}{2}\Delta_0 \right)
	$$
	on $S_0$ is nef. Therefore, it follows that
	$$
	\left(\left(\sigma^*\left(K_{S_0} - \frac{2(d-2)}{3} C_0\right) \right) \cdot \left(2(\pi^*K_X)|_S - (b + d)C \right) \right) \ge 0,
	$$
	i.e.,
	$$
	2 \left( (\pi^*K_X)|_S  \cdot (\sigma^*K_{S_0}) \right) \ge \frac{4(d-2)}{3}((\pi^*K_X)\cdot C)+ (b + d) \left((\sigma^*K_{S_0}) \cdot C \right).
	$$
	Note that $(K_{S_0} \cdot C_0) = 2$ and $((\pi^*K_X)\cdot C)=1$ by Proposition \ref{prop: Noether (1,2)-surface} and Lemma \ref{lem: E0}. Thus the above inequality becomes
	$$
	\left(  (\pi^*K_X)|_S  \cdot (\sigma^*K_{S_0}) \right) \ge \frac{2(d-2)}{3} + (b+d) = b + \frac{5}{3}d - \frac{4}{3}.
	$$
	On the other hand, by \eqref{eq: restriction comparison} and \eqref{eq: step 0}, we have
	$$
	\left(  (\pi^*K_X)|_S  \cdot (\sigma^*K_{S_0}) \right) \le 2 \left((\pi^*K_X)|_S \right)^2 \le 2K_X^3.
	$$
	The above two inequalities imply that
	$$
	b \le 2K_X^3+\frac{4}{3}-\frac{5}{3}d \le 2K_X^3 - \frac{5(p_g(X) - 1)}{3} + 3,
	$$
	where the last inequality holds since $d \ge p_g(X)-2$. It follows that
	$$
	u_1 \le h^0(S, M_1|_S) \le b + 1 \le 2K_X^3 - \frac{5(p_g(X) - 1)}{3} + 4.
	$$
	Again, since $u_1$ is an integer, \eqref{eq: step 2} holds also in this case.
	
	\textbf{Step 3}. In this step, we prove that if $h^0(X', M_2) \ge 2$, then 
	$$
	h^0(X', M_2 - F') > 0.
	$$
	
	Suppose that $h^0(X', M_2) \ge 2$. Then we have $2 \pi^*K_X \ge 2M+M_2$. Thus
	$$
	2(\pi^*K_X)|_{F'}\ge 2M|_{F'}+{M_2}|_{F'}.
	$$
	By Lemma \ref{lem: g2}, $M|_{F'}\sim C$ and $(\pi^*K_X)|_{F'} = (\pi|_{F'})^*K_F$. We deduce that $(((\pi|_{F'})^*K_F) \cdot  M_2|_{F'})=0$. Since $|M_2|$ is base point free, we conclude $M_2|_{F'}=0$ by the Hodge index theorem. Note that we have the exact sequence
	$$
	0 \to H^0(X', M_2 - F') \to H^0(X', M_2) \to H^0(F', M_2|_{F'}).
	$$
	Thus the result follows.

	\textbf{Step 4}. In this step, we finish the whole proof of the proposition.
	
	By \eqref{eq: decomp P_2}, \eqref{eq: step 1} and \eqref{eq: step 2}, we have
	$$
	P_2(X)\le\rounddown{2K_X^3}+\rounddown{2K_X^3 - \frac{5(p_g(X) - 1)}{3}} + 6 + h^0(X', 2K_{X'}-2M).
	$$
	If $h^0(X', 2K_{X'}-2M) = 1$, then the above inequality is identical to the inequality (1) in the proposition.
	
	Suppose that $h^0(X', 2K_{X'}-2M)\ge 2$. By \textbf{Step 3}, $h^0(X', M_2-F') > 0$. Thus we have
	$$
	2\pi^*K_X\ge 2M + M_2 \ge 2M + F'.
	$$
	Note that by \eqref{eq: K_XM^2}, \eqref{eq: K_XMZ} and the fact that $d \ge p_g(X) - 2$, we have
	\begin{align*}
		\left((\pi^*K_X)^2\cdot M \right) & = ((\pi^*K_X) \cdot M^2) + \left( (\pi^*K_X)|_S \cdot Z|_S \right)\\
		& \ge p_g(X) + \frac{1}{2}\roundup{\frac{2 (p_g(X) - 4)}{3}} - 2.
	\end{align*}
	The above two inequalities imply that
	\begin{align*}
		K_X^3 & \ge \left((\pi^*K_X)^2\cdot M \right)+\frac{1}{2} \left((\pi^*K_X)^2 \cdot F' \right) \\
		& \ge p_g(X) + \frac{1}{2} \roundup{\frac{2 (p_g(X) - 4)}{3}} - \frac{3}{2} \\
		& \ge \frac{4}{3} p_g(X) - \frac{17}{6}.
	\end{align*}
	This is the inequality (2) in the proposition. The proof is completed.
\end{proof}

\subsection{A special case} \label{subsection: pencil}
In this subsection, we consider the case when
$$
f_*\omega_X = \CO_{\PP^1} (a) \oplus \CO_{\PP^1}
$$
with $a = p_g(X) - 2$. Since $p_g(X) \ge 3$, we have $a \ge 1$.

By Lemma \ref{lem: nef Hodge bundle}, $\Sigma$ is a cone, and is birational to $\FF_a$. Moreover, $\Sigma$ is normal. Thus $\tau$ is an isomorphism. Now we have the following commutative diagram:
$$
\xymatrix{
	X' \ar@{-->}[rr]^{\psi_0} \ar[rrd]_{\psi = \phi_M} & & \FF_a \ar[rr]^{p} \ar[d] & & \mathbb{P}^1 \\
	& & \Sigma \ar@{-->}[rru] & &
}
$$
where $\FF_a \to \Sigma$ is the blow-up of the cone singularity $v$ of $\Sigma$, $\psi_0$ is induced by this blow-up. Let $\bm{s}$ be the section on $\FF_a$ with $\bm{s}^2 = -a$ and $\bm{l}$ be a ruling on $\FF_a$.

\begin{lemma} \label{lem: psi_0}
	The rational map $\psi_0$ is a morphism. Moreover, $p \circ \psi_0 = f'$ and
	\begin{equation} \label{eq: M}
		|M| = \psi_0^*|\bm{s} + a\bm{l}|.
	\end{equation}
\end{lemma}

\begin{proof}
	This is actually implied by the proof of Lemma \ref{lem: nef Hodge bundle}.
\end{proof}

	
	

Recall that by Lemma \ref{lem: E0},  there is a unique $\pi$-exceptional prime divisor $E_0$ satisfying the condition therein.

\begin{lemma}\label{lem: D0}
	There exists a unique prime divisor $D_0$  such that
	\begin{itemize}
		\item [(1)] $\coeff_{D_0} (\psi_0^*\bm{s}) = 1$;
		\item [(2)] $(D_0 \cdot E_0 \cdot F')=1$ and $((\pi^*K_X) \cdot D_0 \cdot F')=1$.
	\end{itemize}
\end{lemma}

\begin{proof}
	By the abuse of notation, we still denote by $C$ the general fiber of $\psi_0$. By Lemma \ref{lem: g2} and \ref{lem: psi_0}, we have
	$$
	C \equiv M|_{F'} \sim \left(\psi_0^*(\bm{s} + a\bm{l})\right)|_{F'} \equiv (\psi_0^*\bm{s})|_{F'}.
	$$
	Thus by Lemma \ref{lem: E0} (3), $((\psi_0^*\bm{s}) \cdot E_0 \cdot F')= (E_0 \cdot C) =1$. Since $\psi_0^*\bm{s}$ is Cartier and $E_0 \nsubseteq \mathrm{Supp}(\psi_0^*\bm{s})$, for any prime divisor $D$ with $\coeff_{D} (\psi_0^*\bm{s}) > 0$, we have $\coeff_{D}(\psi_0^*\bm{s}) \ge 1$ and $(D \cdot E_0 \cdot F')$ is a non-negative integer. Thus there exists a unique prime divisor $D_0$ with $\coeff_{D_0} (\psi_0^*\bm{s})  = 1$ such that 
	$$
	(D_0 \cdot E_0 \cdot F')=1.
	$$ 
	Note that $(\pi^*K_X)|_{F'}=(\pi|_{F'})^*K_F\equiv C+E_0|_{F'}$. Therefore,
	$$
	\left((\pi^*K_X) \cdot D_0 \cdot F'\right) = (E_0\cdot D_0\cdot F') = 1.
	$$
	The proof is completed.
\end{proof}

\begin{lemma}\label{lem: construction of effective divisor (1,2) surface}
	Let $A$ be an ample Cartier divisor on $\FF_a$. For any integer $m>0$, there exists an integer $c>0$ and an effective divisor $H_m \sim cm(K_{X'/{\FF_a}}+E_0)+ c \psi_0^*A$ such that $E_0\nsubseteq\mathrm{Supp}(H_m)$, where $K_{X'/{\FF_a}}=K_{X'} - \psi_0^*K_{\FF_a}$ and $E_0$ is the $\pi$-exceptional divisor as in Lemma \ref{lem: E0}.
\end{lemma}

\begin{proof}
	The proof of \cite[Claim 4.9]{Chen_Chen_Jiang} works verbatim in our setting, and we only need to replace $W$, $g$, $\FF_a$ and $E_0$ loc. cit. by $X'$, $\psi_0$, $\FF_a$ and $E_0$ in our context.
\end{proof}

\begin{lemma}\label{lem: weak pseudo-effective}
	For any nef $\QQ$-divisor $L$ on $X'$, we have
	$$
	\left((3\pi^*K_X-(a - 2)F') \cdot D_0\cdot L\right) \ge 0,
	$$
	where $D_0$ is the divisor as in Lemma \ref{lem: D0}.
\end{lemma}

\begin{proof}
	The proof is just a slight modification of that of \cite[Claim 4.10]{Chen_Chen_Jiang}. By \eqref{eq: modification}, \eqref{eq: exceptional divisor} and Lemma \ref{lem: psi_0}, we have
	\begin{align}
		K_{X'/{\FF_a}} + E_0 & = (\pi^*K_X + E_\pi) + \psi_0^*\left(2 \bm{s} + (a+2) \bm{l} \right) + E_0 \nonumber \\
		& =  \pi^*K_X + 2M - (a - 2)\psi_0^*\bm{l} + E_\pi + E_0  \label{eq: K+E_0} \\
		& = 3\pi^*K_X - (a - 2)\psi_0^*\bm{l} + E_\pi + E_0 - 2Z. \nonumber
	\end{align}
	Write $E_\pi+E_0-2Z=N_{+}-N_{-}$, where $N_{+}$ and $N_{-}$ are both effective $\mathbb{Q}$-divisors with no common irreducible components. By Lemma \ref{lem: E0} (1), we deduce that $E_0 \nsubseteq \supp(N_{+})$ and $E_0 \nsubseteq \supp(N_{-})$.
	
	Choose an ample divisor $A = at_1\bm{l} + t_2 \bm{s}$ on $\FF_a$, where $t_1 > t_2$ are two positive integers such that $t_2 K_X$ is Cartier. Let $m$ be a positive integer such that $mK_X$ is Cartier. By Lemma \ref{lem: construction of effective divisor (1,2) surface}, there exists an integer $c>0$ and an effective divisor $H_m \sim cm (K_{X'/{\FF_a}}+E_0) + c \psi_0^*A$ such that $E_0\nsubseteq \mathrm{Supp}(H_m)$. Thus we have
	\begin{align*}
		& \hspace{16 pt} H_m+cmN_{-}+ct_2Z \\
		& \sim cm (K_{X'/{\FF_a}}+E_0) + c\psi_0^*A + cmN_{-} + ct_2Z\\
		& = cm \left(3 \pi^*K_X - (a - 2)\psi_0^*\bm{l} + E_\pi + E_0 -2Z + N_{-} \right) + c\psi_0^*A +ct_2Z \\
		& = cm\left(3\pi^*K_X - (a - 2)\psi_0^*\bm{l} \right)  + cm N_+ + c \left( at_1 \psi_0^* \bm{l} + t_2(\psi_0^* \bm{s} + Z) \right) \\
		& = cm \left(3\pi^*K_X - (a - 2)\psi_0^*\bm{l} \right) + cmN_{+}  + c \left(a(t_1-t_2) \psi_0^*\bm{l} + t_2(M+Z) \right) \\
		& = cm \left(3\pi^*K_X - (a - 2)\psi_0^*\bm{l}\right) + cmN_{+} + c\left(a (t_1 - t_2)\psi_0^*\bm{l} + t_2 \pi^*K_X \right).
	\end{align*}
	Here the first equality is by \eqref{eq: K+E_0}, and the last two equalities are by \eqref{eq: M} and \eqref{eq: modification}, respectively. By Lemma \ref{lem: psi_0}, $\phi_0^*\bm{l} = F'$. This implies
	\begin{align*}
		& \hspace{16 pt} H_m+cmN_{-}+ct_2Z - cmN_+ \\
		& \sim cm \pi^*\left(3K_X - (a - 2)F \right) + c \pi^*\left(a(t_1-t_2)F + t_2 K_X\right).
	\end{align*}
	Note that $N_+$ is $\pi$-exceptional. We deduce that $cmN_{+}$ is contained in the fixed part of $|H_m+cmN_{-}+ct_2Z|$. In particular,
	$H_m+cmN_{-}+ct_2Z-cmN_{+}$ is effective.
	
	Let $G_m = \frac{1}{cm}(H_m+cmN_{-}+ct_2Z-cmN_{+})$. Since $E_0 \nsubseteq \mathrm{Supp}(H_m) \cup \mathrm{Supp}(N_{+}) \cup \mathrm{Supp}(N_{-})$, by Lemma \ref{lem: E0} (1), $\coeff_{E_0} (G_m) = \frac{t_2}{m} \coeff_{E_0} (Z) = \frac{t_2}{m}$. By Lemma \ref{lem: D0} (2),
	$$
	\left( \left(G_m-\frac{t_2}{m}E_0 \right)\cdot E_0 \cdot F'\right) \ge \mu_m (D_0 \cdot E_0\cdot F') \ge \mu_m,
	$$
	where $\mu_m = \coeff_{D_0}(G_m)$. Since $E_0$ is $\pi$-exceptional and both $G_m$ and $F'$ are $\pi$-trivial, we have $(G_m \cdot E_0 \cdot F')=0$. It follows that
	$$
	-\frac{t_2}{m}(E_0^2 \cdot F')\ge \mu_m \ge 0.
	$$
	In particular, $\lim\limits_{m \to \infty} \mu_m=0$. Thus for any nef $\QQ$-divisor $L$ on $X'$, we have
	$$
	\lim\limits_{m \to \infty} (G_m\cdot D_0\cdot L) = \lim\limits_{m \to \infty} \left((G_m-\mu_m D_0) \cdot D_0\cdot L\right) \ge 0.
	$$
	By the definition of $G_m$, the above inequality just implies that
	\begin{align*}
		\left((3\pi^*K_X - (a-2) F') \cdot D_0 \cdot L\right) = \lim\limits_{m \to \infty} (G_m\cdot D_0\cdot L) \ge 0.
	\end{align*}	
	The proof is completed.
\end{proof}

\begin{prop} \label{prop: Noether pencil}
	Under the above assumption, if $K_X - bF$ is nef, then we have
	$$
	K_X^3 \ge \frac{4}{3}a  + b - \frac{2}{3}.
	$$
\end{prop}

\begin{proof}
	By \eqref{eq: modification} and \eqref{eq: M}, $\pi^*K_X \equiv aF' + \psi_0^*\bm{s} + Z$. Thus
	$$
	K_X^3 \ge a \left((\pi^*K_X)^2 \cdot F'\right) + \left((\pi^*K_X)^2 \cdot (\psi_0^*\bm{s})\right) \ge a +
	\left((\pi^*K_X)^2 \cdot D_0 \right),
    $$
	where $D_0$ is the unique divisor as in Lemma \ref{lem: D0}. By Lemma \ref{lem: D0} (2) and Lemma \ref{lem: weak pseudo-effective}, we have
	\begin{align*}
		0 & \le \left((3\pi^*K_X - (a-2) F')\cdot D_0 \cdot (\pi^*K_X - bF') \right) \\
		& = 3 \left((\pi^*K_X)^2 \cdot D_0\right) - (a + 3b - 2) \left((\pi^*K_X) \cdot D_0 \cdot F'\right) \\
		& = 3\left((\pi^*K_X)^2 \cdot D_0\right) - (a + 3b - 2).
	\end{align*}
	That is, $((\pi^*K_X)^2 \cdot D_0) \ge \frac{a + 3b - 2}{3}$. Thus it follows that
	$$
	K_X^3 \ge a + \frac{a + 3b - 2}{3} = \frac{4}{3}a + b - \frac{2}{3}.
	$$
	The proof is completed.
\end{proof}

\begin{lemma} \label{lem: 2K-F}
	If $a \ge 21$, then $f_* \CO_X(2K_X)$ is an ample vector bundle over $\PP^1$ of rank $4$. In particular, $f_* \CO_X(2K_X) \otimes \CO_{\PP^1}(-1)$ is nef.
\end{lemma}

\begin{proof}
	Since $P_2(F) =  4$, we deduce that $f_* \CO_X(2K_X)$ is locally free of rank $4$. Taking a birational modification of $X'$, we may assume that $\psi_0^*\bm{s} + Z+E_{\pi}$ is a simple normal crossing $\QQ$-divisor. Let $\mu:= \pi|_{F'}: F' \to F$. By \eqref{eq: modification} and Lemma \ref{lem: psi_0}, we have
	$$
	\mu^*K_F \sim (\pi^*K_X)|_{F'} = M|_{F'} + Z|_{F'} = (\psi_0^*\bm{s} + Z)|_{F'}.
	$$
	Thus we may write $(\psi_0^*\bm{s} + Z)|_{F'} = \mu^*B$ for a divisor $B \sim_{\QQ} K_F$. By a result of J. Koll\'ar \cite[Theorem A.1]{Chen_Chen_Jiang}, $\mathrm{glct}(F) \ge \frac{1}{10}$. Thus the pair $(F, \frac{2}{a}B)$ is klt, and $\mu_*\mathcal{O}_{F'}(K_{F'/F} - \rounddown{\frac{2}{a} \mu^*B}) = \mathcal{O}_F$. It follows that
	\begin{align*}
		h^0 \left(F', K_{F'} + \roundup{\mu^*K_F - \frac{2}{a} \mu^*B}\right) &
		= h^0 \left(F', 2\mu^*K_F + K_{F'/F} - \rounddown{\frac{2}{a} \mu^*B} \right) \\
		& = h^0(F, 2K_F) = 4.
	\end{align*}
    Therefore, $f'_* \CO_{X'} \left(K_{X'} + \roundup{\pi^*K_X - \frac{2}{a} (\psi_0^*\bm{s} + Z)}\right)$ is a vector bundle of rank $4$ over $\PP^1$. Let $F'_1$ and $F'_2$ be two different smooth fibers of $f'$. Consider the divisor
	$$
	D : = \pi^*K_X - \frac{2}{a} (\psi_0^*\bm{s} + Z) - F'_1 - F'_2.
	$$
	Now the fractional part of $D$ is simple normal crossing. By Lemma \ref{lem: psi_0}, $D \sim_{\QQ} (1 - \frac{2}{a}) \pi^*K_X$ is nef and big. Thus by the Kawamata-Viehweg vanishing theorem, we have
	$$
	h^1(X', K_{X'} + \roundup{D}) = 0.
	$$
	By the Leray spectral sequence, the above vanishing implies that
	$$
	h^1\left(\PP^1, f'_* \CO_{X'} \left(K_{X'} + \roundup{\pi^*K_X - \frac{2}{a} (\psi_0^*\bm{s} + Z)} \right) \otimes \CO_{\PP_1}(-2) \right) = 0.
	$$
	We conclude that $f'_* \CO_{X'} \left(K_{X'} + \roundup{(\pi^*K_X - \frac{2}{a} (\psi_0^*\bm{s} + Z)} \right)$ is ample. Since $f'_*$ is left exact, we obtain an inclusion
	$$
	f'_* \CO_{X'} \left(K_{X'} + \roundup{\pi^*K_X - \frac{2}{a} (\psi_0^*\bm{s} + Z)} \right) \hookrightarrow f'_* \omega_{X'}^{\otimes 2} = f_* \CO_X(2K_X)
	$$
	between two vector bundles over $\PP^1$ of the same rank. Therefore, $f_* \CO_X(2K_X)$ is also ample. Thus $f_* \CO_X(2K_X)\otimes \CO_{\PP^1}(-1)$ is nef.
\end{proof}

\begin{prop} \label{prop: 2K-F}
	If $p_g(X) \ge 23$, then $2K_X - F$ is nef. In particular,
	$$
	K_X^3 \ge \frac{4}{3} p_g(X) - \frac{17}{6}.
	$$
\end{prop}

\begin{proof}
	Let $\phi: X \dashrightarrow \PP(f_* \CO_X(2K_X))$ be the relative bicanonical map of $X$ with respect to $f$. Then $\phi$ is induced by the linear system $|2K_X + tF|$ for a sufficiently large integer $t$. Since $F$ is a minimal $(1, 2)$-surface, $|2K_F|$ is base point free \cite[Lemma 2.1]{Horikawa2}. We deduce that the indeterminacy locus of $\phi$ is vertical with respect to $f$. Let $\pi_1: X_1 \to X$ be the blow-up of the indeterminacy locus of $\phi$, and let $\phi_1: X_1 \to \PP(f_* \CO_X(2K_X))$ be the induced morphism by $\pi_1$. Then we have $\phi_1^* H \sim 2 \pi_1^*K_X - E_1$, where $H$ is a relative hyperplane section of $\PP(f_* \CO_X(2K_X))$, and $E_1 \ge 0$ is a vertical $\QQ$-divisor with respect to the fibration $f_1: X_1 \to \PP^1$. Since $a = p_g(X) - 2 \ge 21$, by Lemma \ref{lem: 2K-F}, $f_*\CO_X(2K_X) \otimes \CO_{\PP^1}(-1)$ is nef. Thus the divisor $\pi_1^*(2K_X - F) - E_1$ is also nef.
	
	Suppose that $A$ is an integral curve on $X_1$. If $A$ is horizontal with respect to $f_1$, then
	$$
	\left( \left(\pi_1^*(2K_X - F) \right) \cdot A \right) \ge (E_1 \cdot A) \ge 0.
	$$
	If $A$ is vertical with respect to $f_1$, then
	$$
	\left( \left(\pi_1^*(2K_X - F) \right) \cdot A \right) = 2 \left( (\pi_1^*K_X) \cdot A \right) \ge 0.
	$$
	As a result, $2K_X - F$ is nef. By Proposition \ref{prop: Noether pencil}, we deduce that
	$$
	K_X^3 \ge \frac{4}{3} p_g(X) - \frac{17}{6}.
	$$
	The whole proof is completed.
\end{proof}


\section{Noether inequalities for $3$-folds of general type}

In this section, we establish three Noether inequalities for $3$-folds of general type. Within this section, let $X$ be a minimal $3$-fold of general type with $p_g(X) \ge 3$. Let
$$
\phi_{K_X}: X \dashrightarrow \Sigma \subseteq \PP^{p_g(X) - 1}
$$
be the canonical map of $X$ with the image $\Sigma$.

\subsection{The case when $\dim \Sigma = 3$} We start from the easiest case.

\begin{prop} \label{prop: Noether dim 3}
	Suppose that $\dim \Sigma = 3$. Then
	$$
	K_X^3 \ge 2p_g(X) - 6.
	$$
\end{prop}

\begin{proof}
	This is just \cite[Theorem 2.4]{Kobayashi}.
\end{proof}

\subsection{The case when $\dim \Sigma = 2$} Given a $3$-fold $X$ with at worst canonical singularities, there is an associated basket $B_X$ according to Reid \cite{Reid}. By \cite{Fletcher}, $B_X$ is uniquely determined by $X$. Recall  the Riemann-Roch formula in \cite[Corollary 10.3]{Reid} for $P_2(X)$:
\begin{equation} \label{eq: Riemann-Roch P2}
	P_2(X) = \frac{1}{2}K_X^3+3\chi(\omega_X)+l_2(X).
\end{equation}
Here the correction term
\begin{equation} \label{eq: l2}
	l_2(X) = \sum_{Q} \frac{b_Q(r_Q - b_Q)}{2r_Q},
\end{equation}
where the sum $\sum_{Q}$ runs over all singularities $Q\in B_X$ with the type $\frac{1}{r_Q}(1, -1, b_Q)$ ($b_Q$ and $r_Q$ are coprime, and $0 < b_Q \le \frac{1}{2}r_Q$). In particular, $l_2(X) \ge 0$. Moreover, we have the following facts:
\begin{itemize}
	\item [(1)] $l_2(X) = 0$ if and only if $X$ is Gorenstein. Otherwise, $l_2(X) \ge \frac{1}{4}$.
	\item [(2)] $l_2(X) = \frac{1}{4}$ if and only if $X$ has only one non-Gorenstein terminal singularity, and it is of type $\frac{1}{2}(1, -1, 1)$.
	\item [(3)] $l_2(X) = \frac{1}{2}$ if and only if one of the following two cases occurs:
	\begin{itemize}
		\item $X$ has two non-Gorenstein terminal singularities, and they are of type $\frac{1}{2}(1, -1, 1)$;
		\item $X$ has only one non-Gorenstein terminal singularity, and it is of type $cA_1/\mu_2$.
	\end{itemize}
\end{itemize}
We refer the reader to \cite{Reid} for more details regarding the above formula.

\begin{prop} \label{prop: Noether1 dim 2}
	Suppose that $\dim \Sigma = 2$, $p_g(X) \ge 7$ and $p_g(X) \equiv 1$ $(\mathrm{mod}$ $3)$. Then
	\begin{equation} \label{eq: first Noether}
		K_X^3 \ge \frac{4}{3} p_g(X)-\frac{10}{3}.
	\end{equation}
	If the equality holds, then $X$ is Gorenstein, and it follows that $X$ is factorial. If the equality does not hold, then
	$$
	K_X^3 \ge \frac{4}{3} p_g(X) - \frac{17}{6}.
	$$
\end{prop}

\begin{proof}
	To prove the proposition, we may assume that
	$K_X^3 < \frac{4}{3}p_g(X)-2$. By Proposition \ref{prop: fibered minimal model}, we may further assume that $X$ admits a fibration $f: X \to \PP^1$ such that the general fiber $F$ is a minimal $(1,2)$-surface. Thus all results in \S \ref{section: geometry} apply here.
	
	Since $p_g(X) \equiv 1$ $(\mathrm{mod}$ $3)$, by Proposition \ref{prop: Noether (1,2)-surface} and Remark \ref{rmk: d=pg-2}, we deduce that
	$$
	K_X^3 \ge p_g(X) - 2 + \frac{1}{2} \roundup{\frac{2(p_g(X)-4)}{3}} = \frac{4}{3}p_g(X) - \frac{10}{3}.
	$$
	Thus the inequality \eqref{eq: first Noether} holds.
	
	Now suppose that $K_X^3 = \frac{4}{3} p_g(X) - \frac{10}{3}$. Since $p_g(X) \equiv 1$ $(\mathrm{mod}$ $3)$, $K_X^3$ is an even integer. By Proposition \ref{prop: upper bound of P2} (1), we deduce that
	$$
	P_2(X) \le 4K_X^3 - \frac{5(p_g(X) - 1)}{3}  + 7 = \frac{11}{3}p_g(X) - \frac{14}{3}.
	$$
	On the other hand, by Lemma \ref{lem: H1H2}, we have $\chi(\omega_X)=p_g(X)-1$. Thus the Riemann-Roch formula \eqref{eq: Riemann-Roch P2} for $X$ is equivalent to
	$$
	P_2(X) = \frac{1}{2}K_X^3 + 3(p_g(X) - 1) + l_2(X) =  \frac{11}{3}p_g(X) - \frac{14}{3} + l_2(X).
	$$
	As a result, we have $l_2(X) = 0$. Thus $X$ is Gorenstein. By \cite[Lemma 5.1]{Kawamata}, $X$ is factorial.
	
	In the following, we assume that $K_X^3 > \frac{4}{3} p_g(X) - \frac{10}{3}$. Suppose on the contrary that $K_X^3 < \frac{4}{3}p_g(X) - \frac{17}{6}$. Since $p_g(X) \equiv 1$ $(\mathrm{mod}$ $3)$, we deduce that
	$$
	\rounddown{2K_X^3} = \frac{8}{3}p_g(X) - \frac{20}{3}.
	$$
	By Proposition \ref{prop: upper bound of P2} (1) again, we have
	$$
	P_2(X) \le 2 \rounddown{2K_X^3} - \frac{5(p_g(X) - 1)}{3} + 7 = \frac{11}{3} p_g(X) - \frac{14}{3}.
	$$
	On the other hand, by the Riemann-Roch formula \eqref{eq: Riemann-Roch P2} and Lemma \ref{lem: H1H2}, we deduce that
	$$
	P_2(X) = \frac{1}{2}K_X^3 + 3(p_g(X) - 1) + l_2(X) > \frac{11}{3} p_g(X) - \frac{14}{3}.
	$$
	This is a contradiction. Therefore, we have
	$$
	K_X^3 \ge \frac{4}{3}p_g(X) - \frac{17}{6}.
	$$
	The proof is completed.
\end{proof}

\begin{prop} \label{prop: Noether2 dim 2}
	Suppose that $\dim \Sigma = 2$, $p_g(X) \ge 7$ and $p_g(X) \equiv 2$ $(\mathrm{mod}$ $3)$. Then
	\begin{equation} \label{eq: second Noether}
		K_X^3 \ge \frac{4}{3} p_g(X) - \frac{19}{6}.
	\end{equation}
	If the equality holds, then $X$ has only one non-Gorenstein terminal singularity, and it is of type $\frac{1}{2}(1, -1, 1)$. If the equality does not hold, then
	$$
	K_X^3 \ge \frac{4}{3}p_g(X) - \frac{17}{6}.
	$$
\end{prop}

\begin{proof}
	As in the proof of Proposition \ref{prop: Noether1 dim 2}, we may assume that $K_X^3 < \frac{4}{3}p_g(X) - 2$. Thus all results in \S \ref{section: geometry} apply.
	
	Since $p_g(X) \equiv 2$ $(\mathrm{mod}$ $3)$, we have $\roundup{\frac{2(p_g(X)-4)}{3}} = \frac{2p_g(X) - 7}{3}$. By Proposition \ref{prop: Noether (1,2)-surface} and Remark \ref{rmk: d=pg-2}, we deduce that
	$$
	K_X^3 \ge p_g(X) - 2 + \frac{2p_g(X) - 7}{6} = \frac{4}{3}p_g(X) - \frac{19}{6}.
	$$
	Thus the inequality \eqref{eq: second Noether} holds.
	
	Now suppose that $K_X^3 = \frac{4}{3}p_g(X) - \frac{19}{6}$.
	Since $p_g(X) \equiv 2$ $(\mathrm{mod}$ $3)$, we know that $2K_X^3 = \frac{8}{3}(p_g(X) - 2) - 1$ is an odd integer. Thus $X$ is not Gorenstein. By the Riemann-Roch formula \eqref{eq: Riemann-Roch P2} and Lemma \ref{lem: H1H2}, we have
	$$
	P_2(X) = \frac{1}{2}K_X^3 + 3(p_g(X) - 1) + l_2(X) = \frac{11}{3}(p_g(X) - 2) + \frac{11}{4} + l_2(X).
	$$
	On the other hand, by Proposition \ref{prop: upper bound of P2} (1), we deduce that
	$$
	P_2(X) \le 4K_X^3 - \roundup{\frac{5(p_g(X) - 1)}{3}} + 7 = \frac{11}{3}(p_g(X) - 2) + 3.
	$$
	Therefore, we have $l_2(X) = \frac{1}{4}$. By \eqref{eq: l2}, $X$ has only one non-Gorenstein singularity, and it is of type $\frac{1}{2}(1, -1, 1)$.
	
	In the following, we assume that $K_X^3 > \frac{4}{3}p_g(X) - \frac{19}{6}$. Suppose on the contrary that $K_X^3 < \frac{4}{3}p_g(X) - \frac{17}{6}$. Then $\frac{8}{3}p_g(X) - \frac{19}{3} < 2K_X^3 < \frac{8}{3}p_g(X) - \frac{17}{3}$. Since $p_g(X) \equiv 2$ $(\mathrm{mod}$ $3)$, we deduce that
	$$
	\rounddown{2K_X^3} = \frac{8}{3}p_g(X) - \frac{19}{3}
	$$
	and
	$$
	\rounddown{2K_X^3 - \frac{5(p_g(X) - 1)}{3}} = p_g(X) - 5.
	$$
	Thus by Proposition \ref{prop: upper bound of P2} (1), we have
	$$
	P_2(X) \le \frac{11}{3}p_g(X) - \frac{13}{3} = \frac{11}{3}(p_g(X) - 2) + 3.
	$$
	On the other hand, now $K_X^3$ is not an integer. Thus $X$ is non-Gorenstein. By \eqref{eq: l2}, $l_2(X) \ge \frac{1}{4}$. By the Riemann-Roch formula \eqref{eq: Riemann-Roch P2} and Lemma \ref{lem: H1H2}, we have
	$$
	P_2(X) = \frac{1}{2}K_X^3 + 3 (p_g(X) - 1) + l_2(X) > \frac{11}{3}(p_g(X) - 2) + 3.
	$$
	This is a contradiction. Therefore, we have
	$$
	K_X^3 \ge \frac{4}{3}p_g(X) - \frac{17}{6}.
	$$
	The proof is completed.
\end{proof}

\begin{prop} \label{prop: Noether3 dim 2}
	Suppose that $\dim \Sigma = 2$, $p_g(X) \ge 7$ and $p_g(X) \equiv 0$ $(\mathrm{mod}$ $3)$. Then
	\begin{equation} \label{eq: third Noether}
		K_X^3 \ge \frac{4}{3} p_g(X) - 3.
	\end{equation}
	If the equality holds, then one of the following two cases occurs:
	\begin{itemize}
		\item $X$ has two non-Gorenstein terminal singularities, and they are of the same type $\frac{1}{2}(1, -1, 1)$;
		\item $X$ has only one non-Gorenstein terminal singularity, and it is of type $cA_1/\mu_2$.
	\end{itemize}
\end{prop}

\begin{proof}
	Just as before, we may assume that $K_X^3 < \frac{4}{3}p_g(X) - 2$ so that all results in \S \ref{section: geometry} apply.
	
	Since $p_g(X) \equiv 0$ $(\mathrm{mod}$ $3)$, we have $\roundup{\frac{2(p_g(X)-4)}{3}} = \frac{2p_g(X) - 6}{3}$. By Proposition \ref{prop: Noether (1,2)-surface} and Remark \ref{rmk: d=pg-2}, we deduce that
	$$
	K_X^3 \ge p_g(X) - 2 + \frac{p_g(X) - 3}{3} = \frac{4}{3}p_g(X) - 3.
	$$
	Thus the inequality \eqref{eq: third Noether} holds.
	
	Suppose that $K_X^3 = \frac{4}{3}p_g(X) - 3$. Then $K_X^3$ is an odd integer. Thus $X$ is non-Gorenstein by \cite[$\S$ 2.2]{Chen_Chen_Zhang}. By the Riemann-Roch formula \eqref{eq: Riemann-Roch P2} and Lemma \ref{lem: H1H2},
	$$
	P_2(X) = \frac{1}{2}K_X^3 + 3 (p_g(X) - 1) + l_2(X) = \frac{11}{3}p_g(X) - \frac{9}{2} + l_2(X).
	$$
	By Proposition \ref{prop: upper bound of P2}, we deduce that
	$$
	P_2(X) \le 4K_X^3 - \roundup{\frac{5(p_g(X) - 1)}{3}} + 7 = \frac{11}{3}p_g(X) - 4.
	$$
	As a result, $l_2(X) = \frac{1}{2}$, and the description of the non-Gorenstein singularities on $X$ simply follows from \eqref{eq: l2}.
\end{proof}

\subsection{The case when $\dim \Sigma=1$} We have the following proposition.
\begin{prop} \label{prop: Noether dim 1}
	Suppose that $\dim \Sigma=1$ and $p_g(X) \ge 11$. Then
	$$
	K_X^3 \ge \frac{4}{3}p_g(X)-\frac{8}{3}.
	$$
\end{prop}

\begin{proof}
	Suppose that the connected component of a general fiber of $\phi_{K_X}$ is not a $(1, 2)$-surface. By \cite[Theorem 4.4 and Theorem 4.5]{Chen_Chen_Jiang}, we have
	$$
	K_X^3 > 2p_g(X) - 6 > \frac{4}{3}p_g(X)-\frac{8}{3}.
	$$

	In the following, we assume that the general fiber of $\phi_{K_X}$ is a $(1, 2)$-surface. Since $p_g(X) \ge 11$, by \cite[Corollary 3]{Chen_Chen_Jiang2}, after replacing $X$ by another minimal model, we may assume that there is a fibration $f: X \to B$ over a smooth curve $B$ with general fiber $F$ a minimal $(1,2)$-surface. Moreover, $\phi_{K_X}$ factors through $f$. If $g(B) > 0$, by \cite[Lemma 4.5 (i)]{Chen}, $g(B) = h^1(X, \CO_X) = 1$ and $h^2(X, \CO_X) = 0$. Thus by \cite[Theorem 1.1]{Hu_Zhang}, we have
	$$
	K_X^3 \ge \frac{4}{3} \chi(\omega_X) = \frac{4}{3} p_g(X).
	$$
	
	From now on, we assume that $B \simeq \PP^1$. Since $f_* \omega_{X/\PP^1}$ is nef, we may write
	$$
	f_*\omega_X = \CO_{\PP^1}(a) \oplus \CO_{\PP^1}(-b),
	$$
	where $a = p_g(X) - 1$ and $1 \le b \le 2$. Let $\nu: \bar{B} \to B$ be a double cover branched along $2b$ general points on $B$. Let $\bar{f}: \bar{X} \to \bar{B}$ be the base change of $f$ via $\nu$. Then it is clear that $\bar{X}$ is minimal, and the general fiber $\bar{F}$ of $\bar{f}$ is a minimal $(1, 2)$-surface. Moreover, we have
	$$
	\bar{f}_* \omega_{\bar{X}} = \CL \oplus \CO_{\bar{B}},
	$$
	where $\CL$ is a line bundle on $\bar{B}$ with degree $2a + 2b$. Thus the canonical image of $\bar{X}$ has dimension two. Denote by $\mu: \bar{X} \to X$ the induced double cover. Then we have
	\begin{equation} \label{eq: double cover}
		K_{\bar{X}}^3 = 2(K_X + bF)^3 = 2K_X^3 + 6b.
	\end{equation}
    If $b = 1$, by the Hurwitz formula, we know that $\bar{B} \simeq \PP^1$. Note that $K_{\bar{X}} - 2\bar{F} = \mu^*K_X$ is nef. By Proposition \ref{prop: Noether pencil} and \eqref{eq: double cover}, we have
	$$
	K_X^3 = \frac{1}{2}K_{\bar{X}}^3 - 3 \ge \frac{1}{2} \left(\frac{4}{3}(2a+2) + 2 - \frac{2}{3}\right) - 3 = \frac{4}{3}p_g(X) - \frac{7}{3}.
	$$
	If $b = 2$, by the Hurwitz formula, we know that $g(\bar{B}) = 1$. Now $K_{\bar{X}} - 4\bar{F} = \mu^*K_X$ is nef. By \cite[Remark 2.12]{Hu_Zhang} and \eqref{eq: double cover}, we have
	$$
	K_X^3 = \frac{1}{2}K_{\bar{X}}^3 - 6 \ge \frac{1}{2} \left(\frac{4}{3}(2a+4) + 4\right) - 6 = \frac{4}{3}p_g(X) - \frac{8}{3}.
	$$
	Thus the proof is completed.
\end{proof}

\subsection{Three Noether inequalities} Now we are ready to state the main theorem in this section.

\begin{theorem} \label{thm: Noether lines}
	Let $X$ be a minimal $3$-fold of general type with $p_g(X) \ge 11$.
	\begin{itemize}
		\item [(1)] (First Noether inequality) We have the optimal inequality
		$$
		K_X^3 \ge \frac{4}{3}p_g(X) - \frac{10}{3}.
		$$
		If the equality holds, then $p_g(X) \equiv 1$ $(\mathrm{mod}$ $3)$ and $X$ is Gorenstein. Moreover, it follows that $X$ is  factorial.
		
		\item [(2)] (Second Noether inequality) If $K_X^3 > \frac{4}{3}p_g(X) - \frac{10}{3}$, then we have the optimal inequality
		$$
		K_X^3 \ge \frac{4}{3}p_g(X) - \frac{19}{6}.
		$$
		If the equality holds, then $p_g(X) \equiv 2$ $(\mathrm{mod}$ $3)$. Moreover, $X$ has only one non-Gorenstein terminal singularity, and it is of type $\frac{1}{2}(1, -1, 1)$.
		
		\item [(3)] (Third Noether inequality) If $K_X^3 > \frac{4}{3}p_g(X) - \frac{19}{6}$, then we have the optimal inequality
		$$
		K_X^3 \ge \frac{4}{3}p_g(X) - 3.
		$$
		If the equality holds, then $p_g(X) \equiv 0$ $(\mathrm{mod}$ $3)$. Moreover, one of the following two cases occurs:
		\begin{itemize}
			\item $X$ has two non-Gorenstein terminal singularities, and they are of type $\frac{1}{2}(1, -1, 1)$;
			\item $X$ has only one non-Gorenstein terminal singularity, and it is of type $cA_1/\mu_2$.
		\end{itemize}
	\end{itemize}
\end{theorem}

\begin{proof}
	This is a combination of Proposition \ref{prop: Noether dim 3}, \ref{prop: Noether1 dim 2}, \ref{prop: Noether2 dim 2}, \ref{prop: Noether3 dim 2} and \ref{prop: Noether dim 1}. By the examples in Proposition \ref{prop: example} in \S \ref{subsection: example}, the above three inequalities are all optimal. Thus they are indeed the first three Noether lines.
\end{proof}

\begin{theorem} \label{thm: small volume}
	Let $X$ be a minimal $3$-fold of general type with $p_g(X) \ge 11$ and $K_X^3 < \frac{4}{3} p_g(X) - \frac{8}{3}$. Then the following statements hold:
	\begin{itemize}
		\item [(1)] $h^1(X, \CO_X) = h^2(X, \CO_X) = 0$;
		\item [(2)] The canonical image $\Sigma \subseteq \PP^{p_g(X) - 1}$ of $X$ is a non-degenerate surface of degree $p_g(X) - 2$, and $\Sigma$ is smooth when $p_g(X) \ge 23$;
		\item [(3)] $X$ is simply connected.
	\end{itemize}
\end{theorem}

\begin{proof}
	Since $p_g(X) \ge 11$ and $K_X^3 < \frac{4}{3} p_g(X) - \frac{8}{3}$, by Proposition \ref{prop: Noether dim 3} and \ref{prop: Noether dim 1}, we know that the canonical image of $X$ is a surface. By Proposition \ref{prop: fibered minimal model}, we may assume that $X$ admits a fibration $f: X \to \PP^1$ with general fiber $F$ a $(1, 2)$-surface. Then the statement (1) just follows from Lemma \ref{lem: H1H2}.
	
	By Proposition \ref{prop: fibered minimal model}, the canonical image $\Sigma \subseteq \PP^{p_g(X) - 1}$ of $X$ is a non-degenerate surface of degree $p_g(X) - 2$. Now we have
	$$
	f_* \omega_X = \CO_{\PP^1}(a) \oplus \CO_{\PP^1}(b)
	$$
	with $a \ge b \ge 0$ and $a + b = p_g(X) - 2$. If $p_g(X) \ge 23$, by Proposition \ref{prop: 2K-F}, we have $b > 0$. By Lemma \ref{lem: nef Hodge bundle}, $\Sigma$ is smooth. Thus the statement (2) holds.
	
	For (3), let $\alpha: X_1 \to X$ be a resolution of singularities of $X$ such that $\alpha$ is an isomorphism over the smooth locus of $X$, and let $f_1: X_1 \to \PP^1$ be the induced fibration with a general fiber $F_1$. Recall that $f$ has a natural section $\Gamma$. Since $X$ has only terminal singularities, we conclude that $\Gamma_1 := \alpha_*^{-1} \Gamma$ is also a section of $f_1$. In particular, $f_1$ has no multiple fibers. Now we may use the same argument as in the proof of \cite[Theorem 4.7]{Hu_Zhang} to deduce that
	$$
	\pi_1(X_1) \simeq \pi_1(\PP^1) \simeq \{1\},
	$$
	i.e., $X_1$ is simply connected. Thus by a result of S. Takayama \cite[Theorem 1.1]{Takayama}, $X$ is simply connected.
\end{proof}

\begin{remark}
	Note that if two smooth complex projective varieties are birational to each other, then their fundamental groups are isomorphic. Thus the above proof actually shows that any smooth model of $X$ is simply connected.
\end{remark}


\section{Threefolds on the Noether lines}

In this section, we first classify minimal $3$-folds of general type with $p_g(X) \ge 11$ which attain the Noether equality in Theorem \ref{thm: Noether lines} (1). Then we construct examples of minimal $3$-folds of general type which attain the three Noether equalities in Theorem \ref{thm: Noether lines}, respectively.

\subsection{Classification of $3$-folds on the first Noether line}
Throughout this subsection, let $X$ be a minimal $3$-fold of general type with $p_g(X) \ge 11$ satisfying $K_X^3 = \frac{4}{3}p_g(X) - \frac{10}{3}$. By Theorem \ref{thm: Noether lines} (1), $K_X$ is Cartier and $p_g(X) = 3m - 2$ for an integer $m \ge 5$. By Theorem \ref{thm: small volume}, the canonical image of $X$ is a surface of degree $p_g(X) - 2$. Thus by Proposition \ref{prop: fibered minimal model}, replacing $X$ by another minimal model if necessary, we may assume that $X$ admits a fibration $f: X \to \PP^1$ with a general fiber $F$ a $(1, 2)$-surface.

Just as in \S \ref{subsection: general setting}, we have the following commutative diagram:
$$
\xymatrix{
	& & X' \ar[d]_{\pi} \ar[lld]_{f'} \ar[drr]^{\psi} & &  \\
	\mathbb{P}^1 & & X  \ar@{-->}[rr]_{\phi_{K_X}} \ar[ll]^f  & & \Sigma
}
$$
Here $\phi_{K_X}$ is the canonical map of $X$ with the canonical image $\Sigma$, $\pi$ is a birational modification with respect to $K_X$ as in \S \ref{modification}, $\psi$ is the morphism induced by $|M| = \movable|\pi^*K_X|$, and $f' = f \circ \pi$. Note that by Proposition \ref{prop: fibered minimal model}, $\psi$ has connected fibers.

Let
$$
Z = \pi^*K_X - M, \quad E_{\pi} = K_{X'} - \pi^*K_X.
$$
Since $X$ is Gorenstein and terminal, both $Z$ and $E_{\pi}$ are effective Cartier divisors. Denote by $C$ a general fiber of $\psi$ and by $F'$ a general fiber of $f'$. Recall that $f$ admits a section $\Gamma$ such that $\Gamma \cap F$ is the unique base point of $|K_F|$. Note that $\Gamma \simeq \PP^1$ is a smooth curve.

\subsubsection{More properties of $X$}
We recall some results in the proof of Proposition \ref{prop: Noether (1,2)-surface}. Let $S \in |M|$ be a general member. By Bertini's theorem, $S$ is smooth. Let $\sigma: S \to S_0$ be the contraction to the minimal model of $S$. As in the proof of Proposition \ref{prop: Noether (1,2)-surface}, the natural fibration $\psi|_S: S \to \PP^1$ descends to a fibration $S_0 \to \PP^1$ with a general fiber $C_0 = \sigma_* C$.

Let $E$ be the unique $\pi$-exceptional prime divisor as in Lemma \ref{lem: E0}. Denote $\Gamma_S = E|_S$. By Lemma \ref{lem: E0}, $\Gamma_S$ is a section of $\psi|_S$, and we may write
$$
E_\pi|_S = \Gamma_S + E_V, \quad Z|_S = \Gamma_S + Z_V.
$$
Here both $E_V$ and $Z_V$ are effective divisors on $S$, and $(E_V \cdot C) = (Z_V \cdot C) = 0$.

\begin{lemma} \label{lem: Noether line relation}
	The following statements hold:
	\begin{itemize}
		\item [(1)] $(K_X \cdot \Gamma) = \frac{p_g(X) - 4}{3}$;
		\item [(2)] $\sigma_*(E_V + Z_V) = 0$;
		\item [(3)] $K_{S_0} \equiv 2(p_g(X)-2)C_0 + 2\Gamma_{S_0}$, where $\Gamma_{S_0} := \sigma_* \Gamma_S$;
		\item [(4)] $\sigma^*K_{S_0} \sim_{\QQ} 2(\pi^*K_X)|_S$;
		\item [(5)] $E_V - Z_V$ is an effective divisor on $S$.
	\end{itemize}
\end{lemma}

\begin{proof}
	Since $K_X^3 = \frac{4}{3}p_g(X) - \frac{10}{3}$, $X$ attains the equality in Proposition \ref{prop: Noether (1,2)-surface} (2) in which $d = p_g(X) - 2$. Thus by Proposition \ref{prop: Noether (1,2)-surface}, we have
	$$
	(K_X \cdot \Gamma) = K_X^3 - d = \frac{p_g(X) - 4}{3}.
	$$
	The statements (2)--(4) follow directly from Proposition \ref{prop: Noether equality (1,2)-surface}.
	For (5), note that
	$$
	E_V - Z_V = E_\pi|_S - Z|_S = (K_{X'}+S)|_S - 2 (\pi^*K_X)|_S = K_S - 2 (\pi^*K_X)|_S.
	$$
	Since $K_S \ge \sigma^*K_{S_0}$, by (4), there exist an integer $n > 0$ and an effective divisor $D$ on $S$ such that $nE_V \sim D + nZ_V$. By (2), we have $h^0(S, nE_V) = h^0(S_0, \CO_{S_0}) = 1$. Thus
	$$
	nE_V = D + nZ_V.
	$$
	It follows that $E_V - Z_V = \frac{1}{n}D$ is effective.
\end{proof}

Denote $S_X = \pi_*S$. Then we have the following lemma.
\begin{lemma} \label{lem: Noether line KX}
	The linear system $|K_X|$ has no fixed part. Moreover, the surface $S_X \in |K_X|$ is normal with at worst canonical singularities. Moreover, $\pi|_S: S \to S_X$ factors through $\sigma: S \to S_0$.
\end{lemma}

\begin{proof}
	The proof is just identical to that of \cite[Lemma 4.4 and 4.5]{Hu_Zhang}. Thus we omit the proof here and leave it to the interested reader.
\end{proof}

\begin{prop} \label{prop: Noether line KX}
	The following statements hold:
	\begin{itemize}
		\item [(1)] $\baselocus|K_X| = \Gamma$, and $\Gamma$ lies in the smooth locus of $X$;
		\item [(2)] $S_X = S_0$ is smooth;
		\item [(3)]  Let $\pi_\Gamma: X_{\Gamma} \to X$ be the blow-up of $X$ along $\Gamma$. Then $|\pi_\Gamma^*K_X - E_\Gamma|$ is base point free, where $E_\Gamma$ is the $\pi_\Gamma$-exceptional divisor.
	\end{itemize}
\end{prop}

\begin{proof}
	By Lemma \ref{lem: Noether line KX}, $|K_X|$ has no fixed part. Thus $\baselocus|K_X| = \pi(Z \cap S)$. Recall that $Z|_S = \Gamma_S + Z_V$. By Lemma \ref{lem: E0} (2), $\pi(\Gamma_S)=\Gamma$. Thus to prove (1), we only need to show that $\pi(Z_V) \subseteq \Gamma$.
	
	By Lemma \ref{lem: Noether line relation} (2), $\sigma(Z_V)$ consists of finitely many points on $S_0$. By Lemma \ref{lem: Noether line KX}, $S_0$ is the minimal resolution of $S_X$. We deduce that $\pi(Z_V)$ also consists of finitely many points on $S_X$. Suppose that there is a point $p \in \pi(Z_V)$ with $p \notin \Gamma$. Then we may write
	$$
	Z_V = Z_1 + Z_2,
	$$
	where $Z_1$ and $Z_2$ are effective divisors on $S$, $\pi(Z_1) = p$ and $p \notin \pi(Z_2)$. Now by Lemma \ref{lem: Noether line relation} (3) and (4), we know that
	$$
	(\pi^*K_X)|_S = M|_S + Z|_S \equiv (p_g(X)-2)C + \Gamma_S + Z_1 + Z_2.
	$$
	However, since $(Z_1 \cdot C) = 0$ and $Z_1$ does not intersect $Z_2$ or $\Gamma_S$, we deduce that
	$$
	\left(Z_1 \cdot \left( (p_g(X)-2)C + \Gamma_S + Z_2 \right) \right) = 0.
	$$
	This is a contradiction, because $(\pi^*K_X)|_S$ is nef and big, thus $1$-connected. As a result, $\pi(Z_V) \subset \Gamma$ and $\baselocus|K_X| = \Gamma$.
	
	Take any fiber $T$ of $f$. Then $(\Gamma \cdot T) = 1$. Since every irreducible component of $T$ is Cartier by Proposition \ref{prop: Noether1 dim 2}, there is exactly one irreducible component $T_0$ of $T$ such that $\Gamma \cap T_0 \neq \emptyset$. Moreover, $(\Gamma \cdot T_0) = 1$ and $\coeff_{T_0}(T) = 1$. It implies that $T$ is smooth at the point $\Gamma \cap T_0$, so is $X$. Thus (1) is proved.
	
	To prove (2), suppose that $S_X$ is singular. By Bertini's theorem and (1), the singular locus of $S_X$ is contained in $\Gamma$. Let $p \in \Gamma$ be a singularity on $S_X$. By Lemma \ref{lem: Noether line KX}, we have the minimal resolution $\sigma_0: S_0 \to S_X$ such that $K_{S_0} = \sigma_0^*K_{S_X}$. Let $E_p$ be the exceptional divisor on $S_0$ lying over $p$, and let $\Gamma_{S_0} = (\sigma_0)_*^{-1} \Gamma$ be the strict transform of $\Gamma$. Then we have
	$$
	(K_{S_0} \cdot E_p) = 0, \quad (\Gamma_{S_0} \cdot E_p) > 0.
	$$
	On the other hand, by Lemma \ref{lem: Noether line relation} (3), we have
	$$
	(K_{S_0} \cdot E_p) = 2(p_g(X)-2) (C_0 \cdot E_p) + 2(\Gamma_{S_0} \cdot E_p) \ge 2(\Gamma_{S_0} \cdot E_p).
	$$
	This is a contradiction. As a result, $S_X = S_0$ is smooth, and (2) is proved.
	
	For (3), let $\pi_\Gamma: X_\Gamma \to X$ be the blow-up of $X$ along $\Gamma$ with $E_\Gamma$ the $\pi_\Gamma$-exceptional divisor. By (1) and (2), we have
	$$
	K_{X_\Gamma} = \pi_\Gamma^*K_X + E_\Gamma, \quad \pi_\Gamma^*K_X = M_\Gamma + E_\Gamma,
	$$
	where $M_\Gamma = \pi_\Gamma^*K_X - E_\Gamma$. Let $S_\Gamma \in |M_\Gamma|$ be a general member. It suffices to show that the restricted linear system $|M_\Gamma||_{S_\Gamma}$ on $S_\Gamma$ is base point free. Suppose on the contrary that $\baselocus(|M_\Gamma||_{S_\Gamma}) \ne \emptyset$. Since $S_\Gamma$ is the blow-up of $S_X$ along the divisor $\Gamma$, we know that $\pi_\Gamma|_{S_\Gamma}: S_\Gamma \to S_X = S_0$ is an isomorphism. In particular, $S_{\Gamma}$ admits a fibration $S_{\Gamma} \to \PP^1$ and $(\pi_\Gamma|_{S_\Gamma})^*\Gamma = E_\Gamma|_{S_\Gamma}$. Now
	$$
	M_\Gamma|_{S_\Gamma} = (\pi_\Gamma^*K_X)|_{S_\Gamma} - E_\Gamma|_{S_\Gamma} = (\pi_\Gamma|_{S_\Gamma})^*(K_X|_{S_X} - \Gamma).
	$$
	By the adjunction formula and Lemma \ref{lem: Noether line relation} (3), we also have
	$$
	K_X|_{S_X} - \Gamma = \frac{1}{2} K_{S_0} - \Gamma \equiv (p_g(X) - 2)C_0.
	$$
	Thus $\baselocus(|M_\Gamma||_{S_\Gamma})$ is a vertical divisor with respect to the fibration $S_{\Gamma} \to \PP^1$. On the other hand, by (1), $\baselocus(|M_\Gamma||_{S_\Gamma}) = \left(\baselocus|M_\Gamma|\right)\cap {S_\Gamma} \subseteq E_\Gamma|_{S_\Gamma}$. This is a contradiction. As a result, $|M_\Gamma||_{S_\Gamma}$ is base point free, and (3) is proved.
\end{proof}

\subsubsection{The relative canonical model $X_0$} Let
$$
\epsilon: X \to X_0
$$
be the contraction from $X$ onto its relative canonical model $X_0$ over $\PP^1$. Let $f_0: X_0 \to \PP^1$ be the induced fibration. Let $\Gamma_0 = \epsilon_*\Gamma$. Then $\Gamma_0$ is a section of $f_0$. Let
$$
\pi_0: X'_0 \to X_0
$$
be the blow-up along $\Gamma_0$ with the exceptional divisor $E_0$.

\begin{lemma} \label{lem: Noether line KX0}
	We have $\baselocus|K_{X_0}| = \Gamma_0$, and $\Gamma_0$ lies in the smooth locus of $X_0$. Moreover, $|\pi_0^*K_{X_0} - E_0|$ is base point free.
\end{lemma}

\begin{proof}
	Since $\epsilon^*K_{X_0} = K_X$, by Proposition \ref{prop: Noether line KX} (1), $\epsilon^{-1}(\baselocus|K_{X_0}|) = \baselocus|K_X|= \Gamma$. Thus $\baselocus|K_{X_0}| = \Gamma_0$. Moreover, $\epsilon^{-1}(\Gamma_0) = \Gamma$. Now $\epsilon|_\Gamma: \Gamma \to \Gamma_0$ is an isomorphism. By Zariski's main theorem and the fact that $\epsilon$ is birational, we deduce that $\epsilon$ is an isomorphism over a neighbourhood of $\Gamma$. Thus by Proposition \ref{prop: Noether line KX} (1) again, $\Gamma_0$ lies in the smooth locus of $X_0$. The rest part of the lemma simply follows from Proposition \ref{prop: Noether line KX} (3).
\end{proof}

Let $M_0 = \pi_0^*K_{X_0} - E_0$. By Lemma \ref{lem: Noether line KX0}, we obtain a morphism
$$
\phi: X'_0 \to \Sigma
$$
induced by the linear system $|M_0|$, and $\phi$ has connected fibers. By the abuse of notation, we still denote by $C$ a general fiber of $\phi$. Then by Lemma \ref{lem: g2}, $g(C) = 2$. By Lemma \ref{lem: E0} (3), we have  $((\pi_0^*K_{X_0}) \cdot C) = 1$. Thus $(E_0 \cdot C) =  ((K_{X'_0} - \pi^*K_{X_0}) \cdot C) = 1$. In particular, $\phi|_{E_0}: E_0 \to \Sigma$ is birational. Let $f'_0 := f_0 \circ \pi_0: X'_0 \to \PP^1$ be the induced fibration with a general fiber $F'_0$. Then there is a natural $\PP^1$-bundle structure $f'_0|_{E_0}: E_0 \to \PP^1$ on $E_0$. Thus we may assume that $E_0$ is isomorphic to the Hirzebruch surface $\FF_e$ for some $e \ge 0$. Since $\deg \Sigma = p_g(X) - 2$, by \cite[Theorem 7]{Nagata}, either $\Sigma \simeq \FF_e$ or $\Sigma$ is a cone over a smooth rational curve of degree $p_g(X) - 2$. Let
$$
r: \FF_e \to \Sigma
$$
be the blow-up of the cone singularity of $\Sigma$ if $\Sigma$ is a cone, or be the identity morphism if $\Sigma$ is smooth. Let $p: \FF_e \to \PP^1$ be the natural projection.

\begin{lemma} \label{lem: Noether line fibration}
	The rational map
	$$
	\varphi := r^{-1} \circ \phi: X'_0 \dashrightarrow \mathbb{F}_e
	$$
	is a morphism. In particular, $f'_0 = p \circ \varphi$.
\end{lemma}

\begin{proof}
	The proof is the same as that for Lemma \ref{lem: psi_0}. Thus we leave it to the interested reader.
\end{proof}

\begin{lemma} \label{lem: Noether line flatness}
	The morphism $\varphi$ is flat with all fibers integral.
\end{lemma}

\begin{proof}
	Let $W$ be a fiber of $\varphi$. By Lemma \ref{lem: Noether line fibration}, $E_0$ is a section of $\varphi$. Thus $W \cap E_0$ is a point. Suppose that $W$ contains an irreducible component $W_0$ of dimension two. Then $W_0 \cap E_0 = \emptyset$. Thus
	$$
	\left(K_{X_0}^2 \cdot \left((\pi_0)_*W_0\right) \right) = \left((M_0 + E_0)^2 \cdot W_0 \right) = (M_0^2 \cdot W_0) =0.
	$$
	It is clear that $f'_0(W) = p(\varphi(W))$ is a point on $\PP^1$. Since $f'_0 = f_0 \circ \pi_0$, we deduce that $\pi_0(W_0)$ is contained in a fiber of $f_0$. Since $K_{X_0}$ is $f_0$-ample, the above equality implies that $\dim \pi_0(W_0) \le 1$. This is impossible, because the only $\pi_0$-exceptional divisor is $E_0$. As a result, $\dim W = 1$. Thus the fibers of $\varphi$ are equidimensional. By \cite[Theorem 23.1]{Matsumura}, $\varphi$ is flat.
	
	Note that $K_{X_0}$ is Cartier and we have $((\pi_0^*K_{X_0})\cdot W)= 1$. Suppose that $W$ is reducible. Then there is an irreducible component $W_1$ of $W$ such that $((\pi_0^*K_{X_0}) \cdot W_1) = 0$. Thus $(K_{X_0} \cdot ((\pi_0)_* W_1)) = 0$. Similarly as before, since $\pi_0(W_1)$ is contained in a fiber of $f_0$ and $K_{X_0}$ is $f_0$-ample, we deduce that $\pi_0(W_1)$ is a point, which implies that $W_1 \subset E_0$. This is a contradiction. Therefore, $W$ is irreducible. Since 	$((\pi_0^*K_{X_0})\cdot W)= 1$, $W$ is also reduced.
\end{proof}

\begin{lemma} \label{lem: E}
	Let $\CE = \varphi_* \CO_{X'_0}(2E_0)$. Then $\CE$ is a locally free sheaf of rank two on $\mathbb{F}_e$ .
\end{lemma}

\begin{proof}
Take any fiber $C$ of $\varphi$. Since a general fiber of $\varphi$ is of genus $2$, by Lemma \ref{lem: Noether line flatness}, it follows that $C$ is an integral curve of arithmetic genus $2$. We deduce that $h^1(C,\CO_C)=2$. By Theorem \ref{thm: Noether lines} (1)  and Lemma \ref{lem: Noether line KX0}, $X'_0$ is Cohen-Macaulay and the dualizing sheaf $\omega_{X'_0}$ is invertible. Since $C$ is a fiber of $\varphi$ and $\Sigma$ is smooth, $C$ is Cohen-Macaulay and the dualizing sheaf $\omega_C=\omega_{X'_0}|_C$ is invertible. We deduce that 
$$
h^0(C, K_C)=h^0(C, \omega_C)=h^1(C, \mathcal{O}_C)=2,
$$
where the last equality follows from the Serre duality. On the other hand, note that $K_{X'_0} = \pi_0^*K_{X_0} + E_0 = M_0  + 2E_0$. Thus $h^0(C, 2E_0|_C) = h^0(C, K_{X'_0}|_C) =h^0(C, K_C)= 2$. By Grauert's theorem \cite[III, Corollary 12.9]{Hartshorne}, the result follows.
\end{proof}

\subsubsection{Explicit description of $X'_0$}
Let $Y = \PP(\CE)$ be the $\PP^1$-bundle over $\FF_e$, and denote by $q: Y \to \FF_e$ the natural projection. Since all fibers of $\varphi$ are integral, by \cite[Theorem 3.3]{Catanese_Franciosi_Hulek_Reid}, $|K_W| = |2E_0|_W|$ is base point free for any fiber $W$ of $\varphi$. Thus we obtain a morphism
$$
\rho: X'_0 \to Y
$$
which is just the relative canonical map of $X'_0$ over $\mathbb{F}_e$. Since a general fiber $C$ of $\varphi$ is of genus $2$, we deduce that $\rho$ is a finite morphism of degree two. Let $E_Y  = \rho(E_0)$. Then $E_Y$ is a section of $q$. Thus we have the following commutative diagram
$$
\xymatrix{
	X'_0 \ar[drr]_{\varphi} \ar[d]^{\pi_0}  \ar[rr]^{\rho} & &  Y \ar[d]_{q} & & \\
	X_0 \ar@/_1 pc/[rrrr]_{f_0} & & \mathbb{F}_e \ar[rr]^{p} \ar@/_1 pc/[u]_j & & \PP^1
}
$$
where $j: \mathbb{F}_e \to Y$ corresponds to the section $E_Y$. Since $\rho^*E_Y = 2E_0$, we have $\CE = \varphi_*(\rho^*\CO_Y(E_Y)) = q_*\CO_Y(E_Y)$. Thus $E_Y$ is just the relative hyperplane section of $Y$.

From now on, by the abuse of notation, we identify $E_0$ and $E_Y$ with $\FF_e$ under the isomorphism $\varphi|_{E_0}$ and $q|_{E_Y}$, respectively. Denote by $\bm{l}$ a ruling on $\FF_e$ and by $\bm{s}$ the section on $\FF_e$ with $\bm{s}^2 = -e$. Recall that $r: \FF_e \to \Sigma$ is induced by the linear system $|\bm{s} + (e + k)\bm{l}|$ on $\FF_e$ for some $k \ge 0$ with $e + 2k = p_g(X) - 2$. Thus we have
\begin{equation} \label{eq: Nother line M0}
	M_0 = \varphi^*\left(\bm{s} + (e + k)\bm{l}\right),
\end{equation}
In particular,
\begin{equation} \label{eq: Noether line M0E0}
	M_0|_{E_0} = \bm{s} + (e + k)\bm{l}.
\end{equation}
By Lemma \ref{lem: Noether line relation} (1), we have
\begin{equation} \label{eq: Noether line KX0Gamma0}
	(K_{X_0} \cdot \Gamma_0) = (K_X \cdot \Gamma) = \frac{p_g(X) - 4}{3} = \frac{e+2k-2}{3}.
\end{equation}
Thus $(\pi_0^*K_{X_0})|_{E_0} \sim \frac{e+2k-2}{3} \bm{l}$. It follows that
\begin{equation} \label{eq: Noether line E0E0}
	E_0|_{E_0} = (\pi_0^*K_{X_0} - M_0)|_{E_0} \sim -\bm{s} - \frac{2e + k + 2}{3}\bm{l}.
\end{equation}
For any fiber $C$ of $\varphi$, $\varphi|_C: C\to \varphi(C)\cong\PP^1$ is just the canonical map of $C$. Note that $2E_0|_C\in |K_C|$. Thus we have $\rho^*E_Y = 2E_0$ and $\rho|_{E_0}: E_0 \to E_Y$ is an isomorphism. We deduce that
\begin{equation} \label{eq: Noether line EYEY}
	E_Y|_{E_Y} \sim -2\bm{s} - \frac{4e + 2k + 4}{3}\bm{l}.
\end{equation}

Push forward by $q$ on the following short exact sequence
$$
0 \to \CO_Y \to \CO_Y(E_Y) \to \CO_{E_Y}(E_Y) \to 0.
$$
Since $R^1 q_*\CO_Y = 0$, we obtain
\begin{equation} \label{eq: Noether line ses}
	0 \to \CO_{\FF_e} \to \CE \to j^* \CO_{E_Y}(E_Y) \to 0.
\end{equation}
Since $j: \FF_e \to E_Y$ is an isomorphism, by \eqref{eq: Noether line EYEY}, we know that
\begin{equation}
	\det \CE =  j^* \CO_{E_Y}(E_Y) = \CO_{\FF_e} \left(-2\bm{s} - \frac{4e + 2k + 4}{3}\bm{l}\right).
\end{equation}
Furthermore, since $K_{\FF_e} = -2\bm{s} - (e+2)\bm{l}$, we deduce that
\begin{equation} \label{eq: Noether line KY}
	K_Y \sim -2E_Y - q^*\left( 4 \bm{s} + \frac{7e + 2k + 10}{3}\bm{l} \right).
\end{equation}

Let $B$ be the branch locus of $\rho$. Since $X'_0$ has at worst canonical singularities, $B$ is a reduced divisor on $Y$. Note that $E_Y$ is contained in $B$. Thus we may write
$$
B = E_Y + B',
$$
where $E_Y \nsubseteq B'$. Since $\rho$ is a double cover, there is a divisor $L$ on $Y$ such that $B \sim 2L$.

\begin{lemma} \label{lem: Noether line branch locus}
	We have
	\begin{itemize}
		\item [(1)] $\displaystyle{L \sim 3E_Y + 5q^*\left(\bm{s} + \frac{2e + k + 2}{3} \bm{l} \right)}$;
		\item [(2)] $\displaystyle{B' \sim 5E_Y + 10q^*\left(\bm{s} + \frac{2e + k + 2}{3} \bm{l} \right)}$;
		\item [(3)] $B' \cap E_Y = \emptyset$.
	\end{itemize}
\end{lemma}

\begin{proof}
	By \eqref{eq: Nother line M0}, we have $K_{X'_0} = \varphi^*(\bm{s} + (e+k)\bm{l}) + 2E_0$. Note that $\rho^*E_Y = 2E_0$. From \eqref{eq: Noether line KY}, we deduce that
	$$
	\rho^*L \sim K_{X'_0} - \rho^*K_Y \sim \rho^*\left(3E_Y + 5q^*\left(\bm{s} + \frac{2e + k + 2}{3} \bm{l} \right) \right).
	$$
	Since $\Pic(Y)$ is torsion-free, we obtain the linear equivalence in (1). Since $B' \sim 2L - E_Y$, it is clear that (2) just follows from (1). Finally, by (2) and \eqref{eq: Noether line EYEY}, we deduce that $B'|_{E_Y} \sim 0$. Since $E_Y \nsubseteq B'$, (3) follows.
\end{proof}

\begin{lemma} \label{lem: Noether line split}
	The short exact sequence \eqref{eq: Noether line ses} splits. As a result,
	$$
	\CE = \CO_{\FF_e} \oplus \CO_{\FF_e} \left(-2\bm{s} - \frac{4e + 2k + 4}{3}\bm{l}\right).
	$$
\end{lemma}

\begin{proof}
	The proof is identical to that of \cite[Lemma 5.8]{Hu_Zhang}. Thus we leave it to the interested reader.
\end{proof}

\begin{lemma} \label{lem: Noether line Hodge bundle}
	We have
	$$
	f_* \omega_X = (f_0)_* \omega_{X_0} = \CO_{\PP^1} (k) \oplus \CO_{\PP^1}(e+k).
	$$
	In particular, $\varphi$ coincides with the relative canonical map of $X_0$ over $\PP^1$.
\end{lemma}

\begin{proof}
	By \eqref{eq: Noether line KY} and Lemma \ref{lem: Noether line branch locus}, we have
	$$
	\rho_* \omega_{X'_0} = \omega_Y \oplus \omega_Y(L) = \omega_Y \oplus \CO_Y \left(E_Y + q^*\left(\bm{s} + (e + k)\bm{l}\right)\right).
	$$
	Since $q_* \omega_Y = 0$ and $\CE = q_*\CO_Y(E_Y)$, by Lemma \ref{lem: Noether line split} and the projection formula, we deduce that
	$$
	\varphi_* \omega_{X'_0} = \CO_{\FF_e} \left(\bm{s} + (e + k) \bm{l}\right) \oplus \CO_{\FF_e} \left(-\bm{s} - \frac{e - k + 4}{3}\bm{l}\right).
	$$
	Thus it follows that
	$$
	(f_0)_* \omega_{X_0} = (f'_0)_* \omega_{X'_0} = p_*\CO_{\FF_e} \left(\bm{s} + (e + k)\bm{l}\right) = \CO_{\PP^1}(k) \oplus \CO_{\PP^1}(e + k).
	$$
	Note that $\FF_e = \PP((f_0)_*\omega_{X_0})$. Thus $\varphi$ coincides with the relative canonical map of $X_0$ over $\PP^1$.
\end{proof}

\begin{lemma} \label{lem: Noether line KX0 ample}
	If $p_g(X) \ge 23$, then  $K_{X_0}$ is ample.
\end{lemma}

\begin{proof}
	By Theorem \ref{thm: small volume} (2), the canonical image $\Sigma \simeq \FF_e$ is smooth. Thus by Lemma \ref{lem: nef Hodge bundle} and \ref{lem: Noether line Hodge bundle}, $k\ge 1$. Let $F_0$ be a general fiber of $f_0$. We have
	$$
	\pi_0^*(K_{X_0}-F_0) = M_0 + E_0 - \pi_0^*F_0 = \varphi^*\left( \bm{s} + (e+k-1) \bm{l} \right) + E_0.
	$$
	By \eqref{eq: Noether line KX0Gamma0}, we have
	$$
	\left( \pi_0^*(K_{X_0}-F_0) \right)|_{E_0} = \frac{e+2k-5}{3} \bm{l}.
	$$
	Since $e + 2k = p_g(X) - 2 \ge 21$, we deduce that $(\pi_0^*(K_{X_0}-F_0))|_{E_0}$ is nef. Note that $\varphi^*(\bm{s}+(e+k-1)\bm{l})$ is nef and $E_0$ is irreducible. We conclude that $K_{X_0}-F_0$ is nef. On the other hand, since $K_{X_0}$ is $f_0$-ample, $K_{X_0}+tF_0$ is ample for a sufficiently large integer $t$. Thus $K_{X_0} = \frac{t}{t+1}(K_{X_0}-F_0) + \frac{1}{t+1} (K_{X_0}+tF_0)$ is ample. The proof is completed.
\end{proof}

Now we are ready to state the classification theorem for $3$-folds on the (first) Noether line.

\begin{theorem} \label{thm: Noether line description}
	Let $X$ be a minimal $3$-fold of general type on the Noether line with $p_g(X) = 3m - 2 \ge 11$ and with a fibration $f: X \to \PP^1$ such that a general fiber of $f$ is a $(1, 2)$-surface. Let $X_0$ be the relative canonical model of $X$ over $\PP^1$ with the fibration $f_0: X_0 \to \PP^1$. Then $\baselocus|K_{X_0}| = \Gamma_0$ is a section of $f_0$.
	Let $\pi_0: X'_0 \to X_0$ be the blow-up along $\Gamma_0$. Then the induced fibration $f'_0: X'_0 \to \PP^1$ is factorized as
	$$
	f'_0: X'_0 \stackrel{\rho} \longrightarrow Y \stackrel{q} \longrightarrow \FF_e \stackrel{p}  \longrightarrow \PP^1
	$$
	with the following properties:
	\begin{itemize}
		\item [(i)] the Hirzebruch surface $\FF_e$ is isomorphic to $\PP(f_* \omega_X)$;
		\item [(ii)] $q: Y = \PP(\CO_{\FF_e} \oplus \CO_{\FF_e} (-2\bm{s} - (m + e)\bm{l})) \to \FF_e $ is a $\PP^1$-bundle, where $\bm{s}$ is a section on $\FF_e$ with $\bm{s}^2 = -e$ and $\bm{l}$ is a ruling on $\FF_e$;
		\item [(iii)] $\rho: X'_0 \to Y$ is a flat double cover with the branch locus $B = B_1 + B_2$, where $B_1$ is the relative hyperplane section of $Y$, $B_2 \sim 5B_1 + 5(m+e)q^*\bm{l} + 10 q^*\bm{s}$ and $B_1 \cap B_2 = \emptyset$.
	\end{itemize}
    Moreover, if $p_g(X) \ge 23$, then $X_0$ is the canonical model of $X$.
\end{theorem}

\begin{proof}
	Note that $\frac{4e + 2k + 4}{3} = \frac{p_g(X) + 2}{3} + e = m + e$. Thus the classification is a combination of Lemma \ref{lem: Noether line KX0}, \ref{lem: Noether line branch locus}, \ref{lem: Noether line split} and \ref{lem: Noether line Hodge bundle}. If $p_g(X) \ge 23$, by Lemma \ref{lem: Noether line KX0 ample}, $K_{X_0}$ is ample. Thus $X_0$ is the canonical model of $X$.
\end{proof}

\subsection{Examples of threefolds on the three Noether lines} \label{subsection: example} In this subsection, we construct $3$-folds of general type that satisfy each Noether equality in Theorem \ref{thm: Noether lines}.

Let $\FF_e$ be the Hirzebruch surface for some $e \ge 3$. Denote by $\bm{l}$ a ruling on $\FF_e$ and by $\bm{s}$ the section of $\FF_e$ over $\PP^1$ with $\bm{s}^2 = -e$. Fix an integer $a \ge 2e$. Let $D = 2\bm{s} + a\bm{l}$, and let
$$
p: Y := \PP \left(\CO_{\FF_e} \oplus \CO_{\FF_e}(-D)\right) \to \FF_e
$$
be the $\PP^1$-bundle over $\FF_e$. Denote by $V$ the effective relative hyperplane section of $Y$. Then we have
\begin{equation} \label{eq: example KY}
	K_Y =-2 V + p^*(K_{\FF_e} - D) = -2V - p^*\left(4 \bm{s} + (a+e+2) \bm{l}
	\right).
\end{equation}
\begin{lemma} \label{lem: example bpf}
	The linear system $|V + p^*D|$ is base point free. Moreover, a general member in $|V + p^*D|$ does not intersect $V$.
\end{lemma}

\begin{proof}
	Since $a \ge 2e$, it is clear that $|D|$ is base point free. We only need to prove that the restricted linear system $|V + p^*D| |_V$ on $V$ is base point free. By the definition of $V$, we have $\CO_V(V) = \CO_V(-p^*D)$. Moreover, using the spectral sequence twice, we deduce that
	$$
	h^1(Y, p^*D) = h^1(\FF_e, D) = h^1 \left( \PP^1, \CO_{\PP^1}(a - 2e) \oplus \CO_{\PP^1}(a - e) \oplus \CO_{\PP^1}(a) \right) = 0.
	$$
	Therefore, the restriction map
	$$
	H^0(Y, V + p^*D) \to H^0(V, V + p^*D) = H^0(V, \CO_V)
	$$
	is surjective, which implies that $|V + p^*D| |_V$ is base point free. Note that $\CO_V(V + p^*D)  = \CO_V$. Thus a general member in $|V + p^*D|$ will not intersect $V$.
\end{proof}

Fix another integer $b$ with $2b \ge 5a$. Choose a general member $H \in |5V + p^*(10\bm{s} + 2b\bm{l})| = |5V + 5p^*D + (2b-5a) p^*\bm{l}|$. By Bertini's theorem and Lemma \ref{lem: example bpf}, $H$ is smooth, and the divisor $B = H + V$ is simple normal crossing. Let
$$
\rho: X' \to Y
$$
be the double cover branched along $B$. Set $\psi:= p \circ \rho: X' \to \FF_e$. By \cite[Corollary 2.31 (3) and Proposition 5.20 (3)]{Kollar_Mori}, we deduce that $X'$ has at worst canonical singularities. Moreover, if $2b = 5a$, by Lemma \ref{lem: example bpf}, $H \cap V = \emptyset$, and $X'$ is smooth.

Denote $L = 3V + p^*(5 \bm{s} + b\bm{l})$. Then $2L \sim B$. By \eqref{eq: example KY}, we have
$$
K_{X'} \sim \rho^*(K_Y + L) = \rho^*V + \rho^*\left(\bm{s} + (b - a - e - 2) \bm{l} \right).
$$
For simplicity, we denote $N= \bm{s} + (b-a-e-2) \bm{l}$. Since $2b \ge 5a$ and $a \ge 2e \ge 6$, we have
$$
b - a - e - 2 \ge \frac{3}{2} a - e - 2 \ge 2e - 2.
$$
Thus $N$ is ample. Recall that $V$ is contained in the branch locus of $\rho$. We may write $\rho^*V = 2E$, where $E \simeq \FF_e$ is a section of $\psi$. Let
$$
A = K_{X'} - E \sim E + \rho^*N.
$$

\begin{lemma}\label{lem: example A nef big}
	The divisor $A$ is nef and big. Moreover,
	$$
	A^3=3b-\frac{7}{2}a-4e-6.
	$$ 	
\end{lemma}
\begin{proof}
	Note that $\rho|_E: E \to V$ is an isomorphism,  both of which are isomorphic to $\FF_e$. Since $\CO_{E}(2E) = (\rho|_E)^*\CO_{V}(V)$, under the above isomorphism, we have
	\begin{equation} \label{eq: example E2}
		E|_E  \equiv \frac{1}{2} V|_V \equiv  - \bm{s} - \frac{1}{2} a \bm{l}.
	\end{equation}
	Thus
	\begin{equation} \label{eq: example AE}
		A|_E \equiv E|_E + (\rho^*N)|_E \equiv \left( b - \frac{3}{2}a - e -2 \right) \bm{l}.
	\end{equation}
	Since $b - \frac{3}{2}a - e -2 \ge a - e - 2 \ge e-2 > 0$, we deduce that $A|_E$ is nef.
	
	Let $\Gamma$ be an integral curve on $X'$. If $\Gamma \subseteq E$, by the nefness of $A|_E$, we have $(A \cdot \Gamma) \ge 0$. If $\Gamma \nsubseteq E$, by the ampleness of $N$, we have $(A \cdot \Gamma) \ge ((\rho^*N) \cdot \Gamma) \ge 0$. Thus $A$ is nef.
	
	By \eqref{eq: example AE}, $(A^2 \cdot E) = (A|_E)^2 = 0$. Thus
	$$
	A^3 = \left(A^2 \cdot (\rho^*N)\right) = \left(E^2 \cdot (\rho^*N)\right) + 2\left(E \cdot (\rho^*N)^2\right).
	$$
	By \eqref{eq: example E2}, we have
	$$
	\left(E^2 \cdot (\rho^*N)\right) = \left(-\bm{s} - \frac{1}{2} a \bm{l}\right) \left(\bm{s} + (b-a-e-2) \bm{l}\right) = -b + \frac{1}{2}a + 2e + 2,
	$$
	and
	$$
	\left(E \cdot (\rho^*N)^2\right) = N^2 = \left(\bm{s} + (b-a-e-2) \bm{l}\right)^2 = 2b-2a-3e-4.
	$$
	Put the above equalities together, and it follows that
	$$
	A^3 = 3b - \frac{7}{2}a - 4e - 6.
	$$
	Since $3b - \frac{7}{2}a - 4e - 6 \ge 4a - 4e - 6 \ge 4e - 6 > 0$, we conclude that $A$ is big.
\end{proof}

\begin{lemma} \label{lem: example exceptional locus}
	Let $\Gamma$ be an integral curve on $X'$. Then $(A \cdot \Gamma)=0$ if and only if $\Gamma$ is a ruling on $E$. Moreover, if $\Gamma$ is a ruling on $E$, then $(K_{X'}\cdot\Gamma)=(E\cdot\Gamma)=-1$.
\end{lemma}
\begin{proof}
	Suppose that $\Gamma$ is a ruling on $E$. By \eqref{eq: example AE}, $(A \cdot \Gamma) = 0$. By \eqref{eq: example E2}, $(E \cdot \Gamma) = (E|_E \cdot \Gamma) = -1$. Thus $(K_{X'} \cdot \Gamma) = (A \cdot \Gamma) + (E \cdot \Gamma) = -1$.
	
	Now suppose that $(A \cdot \Gamma)=0$. If $\Gamma$ is vertical with respect to $\psi$, then $(A \cdot \Gamma) = (E \cdot \Gamma) > 0$. Thus $\Gamma$ is horizontal with respect to $\psi$. Note that $N$ is ample, which implies that $(E \cdot \Gamma) = (A \cdot \Gamma) - ((\phi^*N) \cdot \Gamma) < 0$. We conclude that $\Gamma \subseteq E$. By \eqref{eq: example AE} again, we deduce that $\Gamma$ is a ruling on $E$.
\end{proof}

\begin{lemma}\label{lem: example pg}
	We have
	$$
	p_g(X')=2b-2a-3e-2.
	$$
\end{lemma}
\begin{proof}
	Recall that $K_{X'} \sim \rho^*(V + p^*N)$. By the projection formula, we have
	$$
	\rho_*\omega_{X'} = \CO_Y(V + p^*N) \oplus \CO_Y(V + p^*N - L).
	$$
	Note that $h^0(Y, V + p^*N - L) = 0$. Thus $p_g(X') = h^0(Y, V + p^*N)$. By the projection formula again, we have
	$$
	p_*\CO_Y(Y, V + p^*N) = \CO_{\FF_e}(N) \oplus \CO_{\FF_e}(N - D).
	$$
	Since $h^0(\FF_e, N-D) = 0$, we deduce that $p_g(X') = h^0(\FF_e, N)$. Finally, by the projection formula once more, we deduce that
	\begin{align*}
		h^0(\FF_e, N) & = h^0 \left(\PP^1, \CO_{\PP^1}(b - a - e - 2)\right) + h^0 \left(\PP^1, \CO_{\PP^1}(b - a - 2e - 2) \right) \\
		& = 2b - 2a - 3e - 2.
	\end{align*}
	Thus the result follows.
\end{proof}

By Lemma \ref{lem: example exceptional locus}, $K_{X'}$ is not nef. Let $\pi: X' \to X$ be the contraction of a $K_{X'}$-negative extremal ray $R$.

\begin{lemma}\label{lem: example construction}
	The morphism $\pi$ is just the contraction of the rulings on $E$. Moreover, $X$ has at worst canonical singularities, and $K_X$ is nef and big. In particular, we have
	\begin{equation} \label{eq: example volume}
		K_X^3=3b-\frac{7}{2}a-4e-6,
	\end{equation}
	and
	\begin{equation} \label{eq: example pg}
		p_g(X)=2b-2a-3e-2.
	\end{equation}
\end{lemma}
\begin{proof}
	Let $\Gamma$ be an integral curve which generates the $K_{X'}$-negative extremal ray $R$, i.e., $R = \mathbb{R}_{\ge 0}[\Gamma]$, where $[\Gamma]$ is the numerical class of $\Gamma$. Since $K_{X'} = A + E$ and $A$ is nef by Lemma \ref{lem: example A nef big}, we have $(E \cdot \Gamma) < 0$. Thus $\Gamma \subseteq E$.
	
	Since $E\cong \mathbb{F}_e$, we may write $\Gamma \equiv m\bm{s} + n \bm{l}$, where $m, n \ge 0$. Thus $[ms]+[nl]=[\Gamma]\in R$. Since $R$ is extremal, we have $[m\bm{s}]\in R$ and $[n\bm{l}]\in R$. By \eqref{eq: example E2} and \eqref{eq: example AE},
	$$
	(K_{X'} \cdot \bm{s}) = (A|_E \cdot \bm{s}) + (E|_E \cdot \bm{s}) = b - 2a - 2 \ge \frac{1}{2} a - 2 \ge e - 2 > 0.
	$$
	Thus $m = 0$, and $R$ is generated by $\bm{l}$. As a result, $\pi$ is just the contraction of the rulings on $E$.
	
	Since $\pi$ is a divisorial contraction, by \cite[Corollary 3.43 (3)]{Kollar_Mori}, $X$ has at worst canonical singularities. Thus we may write
	$$
	K_{X'}=\pi^*K_X+\lambda E,
	$$
	where $\lambda \ge 0$. By Lemma \ref{lem: example exceptional locus}, $-\lambda = \lambda (E \cdot \bm{l}) = (K_{X'} \cdot \bm{l}) = -1$. Thus $\lambda = 1$. It implies that $\pi^*K_X = K_{X'} - E = A$. By Lemma \ref{lem: example A nef big}, we conclude that $K_X$ is nef and big. The equalities \eqref{eq: example volume} and \eqref{eq: example pg} follow immediately from Lemma \ref{lem: example A nef big} and \ref{lem: example pg}.
\end{proof}

\begin{prop} \label{prop: example}
	Let $e \ge 3$ and $k \ge 0$ be two integers.
	\begin{itemize}
		\item [(1)] There exists a minimal $3$-fold $X_{e,k}$ of general type with $p_g(X_{e,k}) = 3e + 6k - 2$ and $K_{X_{e,k}}^3 = 4e + 8k - 6$. In particular,
		$$
		K_{X_{e,k}}^3 = \frac{4}{3}p_g(X_{e,k}) - \frac{10}{3}.
		$$
		\item [(2)] There exists a minimal $3$-fold $X_{e,k}$ of general type with $p_g(X_{e,k}) = 3e + 6k + 2$ and $K_{X_{e,k}}^3 = 4e + 8k - \frac{1}{2}$. In particular,
		$$
		K_{X_{e,k}}^3 = \frac{4}{3}p_g(X_{e,k}) - \frac{19}{6}.
		$$
		\item [(3)] There exists a minimal $3$-fold $X_{e,k}$ of general type with $p_g(X_{e,k}) = 3e + 6k$ and $K_{X_{e,k}}^3 = 4e + 8k - 3$. In particular,
		$$
		K_{X_{e,k}}^3 = \frac{4}{3}p_g(X_{e,k}) - 3.
		$$
	\end{itemize}
\end{prop}

\begin{proof}
	For the $3$-fold in (1), let $a = 2(e+k)$ and $b = 5(e+k)$. By Lemma \ref{lem: example construction}, the above construction for $a$ and $b$ gives rise to a minimal $3$-fold $X$ with $K_X^3 = 4e + 8k - 6$ and $p_g(X) = 3e + 6k - 2$. For the $3$-fold in (2) (resp. (3)), we may take $a = 2(e + k) + 1$ and $b = 5(e + k) + 3$ (resp. $a = 2(e + k)$ and $b = 5(e + k) + 1$). Thus the proof is completed.
\end{proof}

\bibliography{Ref_3foldsmallvolume}
\bibliographystyle{amsplain}

\end{document}